\documentclass{amsart}

\usepackage{hyperref}
\usepackage{amsmath}
\usepackage{amsrefs}
\usepackage[all]{xy}

\date{3 May 2013}

\newtheorem{thm}{Theorem}[section]
\numberwithin{equation}{section}

\newtheorem{lemma}[thm]{Lemma}
\newtheorem{prop}[thm]{Proposition}

\theoremstyle{definition}
\newtheorem{definition}[thm]{Definition}

\theoremstyle{remark}
\newtheorem{remark}[thm]{Remark}

\def\mathcs{C^{*}}
\newcommand{\cs}{\ensuremath{\mathcs}}

\DeclareMathSymbol{\rtimes}{\mathbin}{AMSb}{"6F}
\newcommand{\ib}{im\-prim\-i\-tiv\-ity bi\-mod\-u\-le}

\newcommand{\sme}{\,\mathord{\mathop{\text{--}}\nolimits_{\relax}}\,}
\def\ibind#1{\mathop{#1\mathord{\mathop{\text{--}}}}\!\Ind\nolimits}

\def\C{\mathbf{C}}
\def\T{\mathbf{T}}

\DeclareMathOperator{\Res}{Res}

\DeclareMathOperator{\Ind}{Ind}

\DeclareMathOperator{\Prim}{Prim}
\DeclareMathOperator*{\supp}{supp}

\DeclareMathOperator{\Aut}{Aut}
\def\set#1{\{\,#1\,\}}
\def\sset#1{\{#1\}}
\let\tensor=\otimes
\def\restr#1{|_{{#1}}}

\newbox\hidebox
\def\spechide#1{\setbox\hidebox=\hbox{$#1$}
\hbox to\wd\hidebox{$\box\hidebox^\wedge$\hss}}
\makeatletter
\def\labelenumi{\textnormal{(\@alph\c@enumi)}}
\def\theenumi{\@alph \c@enumi}
\def\labelenumii{\textnormal{(\@roman\c@enumii)}}
\def\theenumii{\@roman \c@enumii}
\newcount\charno
\def\alphapart#1{\charno=96
\advance\charno by#1\char\charno}

\makeatother
\def\<{\langle}
\def\>{\rangle}
\let\ipscriptstyle=\scriptscriptstyle
\def\lipsqueeze{{\mskip -3.0mu}}
\def\ripsqueeze{{\mskip -3.0mu}}
\def\ipcomma{\nobreak\mathrel{,}\nobreak}
\newbox\ipstrutbox
\setbox\ipstrutbox=\hbox{\vrule height8.5pt
width 0pt}
\def\ipstrut{\copy\ipstrutbox}
\def\lip#1<#2,#3>{\mathopen{\relax_{\ipstrut\ipscriptstyle{
#1}}\lipsqueeze
\langle} #2\ipcomma #3 \rangle}
\def\blip#1<#2,#3>{\mathopen{\relax_{\ipstrut
\ipscriptstyle{ #1}}\lipsqueeze\bigl\langle} #2\ipcomma #3 \bigr\rangle}
\def\rip#1<#2,#3>{\langle #2\ipcomma #3
\rangle_{\ripsqueeze\ipstrut\ipscriptstyle{#1}}}
\def\brip#1<#2,#3>{\bigl\langle #2\ipcomma #3
\bigr\rangle_{\ripsqueeze\ipstrut\ipscriptstyle{#1}}}
\def\angsqueeze{\mskip -6mu}
\def\smangsqueeze{\mskip -3.7mu}
\def\trip#1<#2,#3>{\langle\smangsqueeze\langle #2\ipcomma #3
\rangle\smangsqueeze\rangle_{\ripsqueeze\ipstrut\ipscriptstyle{#1}}}
\def\btrip#1<#2,#3>{\bigl\langle\angsqueeze\bigl\langle #2\ipcomma
#3
\bigr\rangle
\angsqueeze\bigr\rangle_{\ripsqueeze\ipstrut\ipscriptstyle{#1}}}
\def\tlip#1<#2,#3>{\mathopen{\relax_{\ipstrut\ipscriptstyle{
#1}}\lipsqueeze \langle\smangsqueeze\langle} #2\ipcomma #3
\rangle\smangsqueeze\rangle}
\def\btlip#1<#2,#3>{\mathopen{\relax_{\ipstrut\ipscriptstyle{
#1}}\lipsqueeze
\bigl\langle\angsqueeze\bigl\langle} #2\ipcomma #3
\bigr\rangle\angsqueeze\bigr\rangle}

\def\ip(#1|#2){(#1\mid #2)}
\def\bip(#1|#2){\bigl(#1 \mid #2\bigr)}
\def\Bip(#1|#2){\Bigl( #1 \bigm| #2 \Bigr)}
\IfFileExists{mathrsfs.sty}{\usepackage{mathrsfs}}
{\let\mathscr\mathcal}
\newcommand{\bundlefont}[1]{\mathscr{#1}}
\newcommand{\A}{\bundlefont A}
\newcommand{\B}{\bundlefont B}

\newcommand\E{\bundlefont E}
\newcommand{\go}{G^{(0)}}
\def\sa_#1(#2,#3){\Gamma_{#1}(#2;#3)}
\newcommand\prima{\Prim A}
\let\phi\varphi

\renewcommand\H{\mathcal{V}}
\newcommand\HH{\mathcal{H}}
\newcommand\WW{\mathcal{W}}
\newcommand\half{\frac12}
\newcommand\mug{\mu}
\newcommand\muh{\nu}
\newcommand\mugmh{\bar\mu}
\newcommand\indgh{\Ind_{H}^{G}}
\newcommand\X{\mathsf{X}}
\newcommand\xind{\ibind{\X}}
\newcommand\pr{\operatorname{pr}}
\newcommand\VV{\mathscr{V}}
\newcommand\pihat{\hat\pi}
\newcommand\Pihat{\hat\Pi}
\newcommand\I{\mathscr{I}}

\def\charfcn#1{\mathbf{1}_{#1}}

\newcommand\indgug{\Ind_{G(u)}^{G}}

\newcommand\Rip{\rip \scriptstyle\star}
\newcommand\Lip{\lip\scriptstyle\star}

\expandafter\ifx\csname BibSpec\endcsname\relax\else
\BibSpec{collection.article}{%
    +{}  {\PrintAuthors}                {author}
    +{,} { \textit}                     {title}
    +{.} { }                            {part}
    +{:} { \textit}                     {subtitle}
    +{,} { \PrintContributions}         {contribution}
    +{,} { \PrintConference}            {conference}
    +{}  {\PrintBook}                   {book}
    +{,} { }                            {booktitle}
    +{,} { }                            {series}
    +{,} { \voltext}                    {volume}
    +{,} { }                            {publisher}
    +{,} { }                            {organization}
    +{,} { }                            {address}
    +{,} { \PrintDateB}                 {date}
    +{,} { pp.~}                        {pages}
    +{,} { }                            {status}
    +{,} { \PrintDOI}                   {doi}
    +{,} { available at \eprint}        {eprint}
    +{}  { \parenthesize}               {language}
    +{}  { \PrintTranslation}           {translation}
    +{;} { \PrintReprint}               {reprint}
    +{.} { }                            {note}
    +{.} {}                             {transition}
}
\BibSpec{article}{%
    +{}  {\PrintAuthors}                {author}
    +{,} { \textit}                     {title}
    +{.} { }                            {part}
    +{:} { \textit}                     {subtitle}
    +{,} { \PrintContributions}         {contribution}
    +{.} { \PrintPartials}              {partial}
    +{,} { }                            {journal}
    +{}  { \textbf}                     {volume}
    +{}  { \PrintDatePV}                {date}
    +{,} { \eprintpages}                {pages}
    +{,} { }                            {status}
    +{,} { \PrintDOI}                   {doi}
    +{,} { available at \eprint}        {eprint}
    +{}  { \parenthesize}               {language}
    +{}  { \PrintTranslation}           {translation}
    +{;} { \PrintReprint}               {reprint}
    +{.} { }                            {note}
    +{.} {}                             {transition}
}
\BibSpec{book}{%
    +{}  {\PrintPrimary}                {transition}
    +{,} { \textit}                     {title}
    +{.} { }                            {part}
    +{:} { \textit}                     {subtitle}
    +{,} { \PrintEdition}               {edition}
    +{}  { \PrintEditorsB}              {editor}
    +{,} { \PrintTranslatorsC}          {translator}
    +{,} { \PrintContributions}         {contribution}
    +{,} { }                            {series}
    +{,} { \voltext}                    {volume}
    +{,} { }                            {publisher}
    +{,} { }                            {organization}
    +{,} { }                            {address}
    +{,} { pp.~}                        {pages}
    +{,} { \PrintDateB}                 {date}
    +{,} { }                            {status}
    +{}  { \parenthesize}               {language}
    +{}  { \PrintTranslation}           {translation}
    +{;} { \PrintReprint}               {reprint}
    +{.} { }                            {note}
    +{.} {}                             {transition}
}
\fi

\allowdisplaybreaks

\begin{document}

\title[Irreducible Representations]{\boldmath Irreducible Induced
  Representations of Fell Bundle \cs-Algebras}

\author{Marius Ionescu}

\address{Department of Mathematics \\ Colgate University,
  Hamilton, NY 13346}

\email{mionescu@colgate.edu}

\author{Dana P. Williams}
\address{Department of Mathematics \\ Dartmouth College \\ Hanover, NH
03755-3551}

\email{dana.williams@Dartmouth.edu}

\thanks{Both Authors were supported by individual grants from the
  Simons Foundation.}  \thanks{Dana would also like to thank Marius and his
  colleagues at Colgate for a very pleasant and productive visit.}

\begin{abstract}
We give precise conditions under which irreducible representations
associated to stability groups induce to irreducible representations
for Fell bundle \cs-algebras.  This result generalizes an earlier
result of Echterhoff 
and the second author.  Because the Fell bundle construction subsumes
most other examples of \cs-algebras constructed from dynamical
systems, our result percolates down to many different constructions
including the many flavors of groupoid crossed products.
\end{abstract}

\maketitle

\section{Introduction}
\label{sec:intro}

One of the fundamental tasks in any study of the ideal structure of
\cs-algebras associated to dynamical systems is to
construct a suitably large class of irreducible representations.  In
the type~I case, it makes sense to try to find representatives of all
equivalence classes of irreducible representations (for example, see
\cite{wil:crossed}*{Theorem~8.16}).  In general, in the presence of
suitable amenability, it is only reasonable to try to construct enough
irreducible representations to account for all primitive ideals.  The
quintessential example is the Gootman-Rosenberg-Sauvagoet Theorem (see
\cite{wil:crossed}*{\S8.3} for a precise statement and further
references): the GRS-Theorem says that for a separable \cs-dynamical
system $\alpha:G\to\Aut A$ with $G$ amenable, every primitive ideal in
$A\rtimes_{\alpha}G$ is induced in an appropriate sense from a
stability group $G_{P}=\set{s\in G:s\cdot P=P}$ with respect to the
induced action of $G$ on $\prima$ for some $P\in\prima$.

Motivated in part by the GRS-Theorem and by results in the case where
the action of $G$ on $\prima$ is smooth, Echterhoff and the second
author have conjectured that every separable \cs-dynamical system
$(A,G,\alpha)$ satisfies the \emph{Effros-Hahn Induction Property}
(EHI) which asserts that if $P\in\prima$ and $J\in
\Prim(A\rtimes_{\alpha}G_{P})$ with $\Res J=P$, then
$\Ind_{G_{P}}^{G}J$ is a primitive ideal in $A\rtimes_{\alpha}G$. (See
\cite{echwil:tams08}*{\S2} for precise definitions and additional
details.)  Although the validity of the conjecture is open in general
--- even when $G$ is amenable and the GRS Theorem holds --- it was
shown in \cite{echwil:tams08} to hold in a wide variety of cases
including all separable systems with $A$ of type~I.  Moreover, in all
cases in which the conjecture is known to hold, a stronger property
holds, called strong-EHI, which asserts that if $\rho\rtimes \pi$ is
an irreducible representation of $A\rtimes_{\alpha}G_{P}$ with
$\ker\rho=P$, then $\Ind_{G_{P}}^{G} (\rho\rtimes\pi)$ is an
irreducible representation of $A\rtimes_{\alpha}G$.  The key
observation in \cite{echwil:tams08} concerning group \cs-dynamical
systems is that strong-EHI always holds if, \emph{in addition}, $\rho$
is assumed to be a \emph{homogeneous} representation (as defined in,
for example, \cite{wil:crossed}*{Definition~G.1}) --- this is the
content of \cite{echwil:tams08}*{Theorem~1.7}.

Our goal here is to extend the results on inducing irreducible
representations in \cite{echwil:tams08}, and
\cite{echwil:tams08}*{Theorem~1.7} in particular, to other sorts of
dynamical system constructions built not only on groups but on
groupoids.  At first glance, there are a horrifying number of
potential targets for such an analysis.  For example, there are
groupoid \cs-algebras with or without a cocycle, and more generally,
one could consider the \cs-algebras associated to twists over
groupoids (also called $\T$-groupoids).  There are also Green twisted
dynamical systems, groupoid dynamical systems and even twisted
versions of groupoid dynamical systems to name a few of the most
important.  Fortunately, as described in detail in
\cite{muh:cm01}*{\S3} or \cite{muhwil:dm08}*{\S2}, all these variants
are subsumed using the \cs-algebra of a separable Fell bundle $p:\B\to
G$ over a locally compact groupoid $G$ with a Haar system.  In this
event, the sections $A=\sa_{0}(\go ,\B)$ form a \cs-algebra and the
groupoid $G$ acts continuously on $\prima$.  Any representation $L$ of
$\cs(G_{P},\B)$ is associated to a representation $\pi$ of the
\cs-algebra $A$.  Our main theorem (Theorem~\ref{thm-main-1.7}) says
that if $L$ is irreducible, $\ker \pi=P$ and $\pi$ is homogeneous,
then $\Ind_{G_{P}}^{G}L$ is irreducible.  This result extends
\cite{echwil:tams08} and we will illustrate how it ``trickles down''
to other dynamical systems settings in Section~\ref{sec:examples}.

Our proof requires an intermediate result which is of considerable
interest on its own.  Namely if $p:\B\to G$ is a separable Fell bundle
over a locally compact groupoid $G$ with Haar system, then we show
that if $u\in \go$, if $G(u)=\set{x\in G:r(x)=u=s(x)}$ is the
stability group of $u$ in $G$ and if $L$ is an irreducible
representation of $\cs(G(u),\B)$, then $\Ind_{G(u)}^{G} L$ is an
irreducible representation of $\cs(G,\B)$ (Theorem~\ref{thm:indirr}).
This result is a direct generalization of
\cite{ionwil:pams08}*{Theorem~5} where the result is proved for
groupoid \cs-algebras (so that $\B$ is the trivial bundle
$\B=G\times\C$).  In fact the proof is disarmingly similar to that in
\cite{ionwil:pams08}, but extra care must be taken to account for the
rather significant difference between scalar-valued sections of a
trivial bundle and Banach space-valued sections of potentially highly
nontrivial Fell bundles.  Combining Theorem~\ref{thm:indirr} with the
usual induction in stages allows us to reduce the proof of our main
theorem to the more comforting setting of a Fell bundle over a group
(rather than a groupoid).

Our paper is organized as follows.  We start in Section~\ref{sec:step-i}
with a very brief review of induced representations of Fell bundle
\cs-algebras and prove our generalization, Theorem~\ref{thm:indirr},
of \cite{ionwil:pams08}*{Theorem~5}. In Section~\ref{sec:red-group-case} we
give the precise definition of the strong Effros-Hahn Induction
property in the Fell bundle setting.  In Section~\ref{sec:group-case}, we
give our proof of the Main Theorem taking advantage of induction in
stages and Theorem~\ref{thm:indirr} to
reduce to the case that $G$ is 
a group.  In Section~\ref{sec:type-i-case} we see that the additional
hypothesis of homogeneity is automatically satisfied in the case that
points are locally closed in the \cs-algebra $A=\sa_{0}(\go,\B)$
associated to the Fell bundle $p:\B\to G$.  While this includes many
interesting classes of algebras, it in particular applies any time $A$
is of type~I.  In Section~\ref{sec:examples} we examine how the Fell
bundle result applies to the examples
of groupoid dynamical systems and their twisted counterparts.  It is
worth noting that special cases of the latter include Green twisted
systems in the case where $G$ is a group, and the \cs-algebras of
twists or $\T$-goupoids when $\B$ is a trivial line bundle.

\subsubsection*{Assumptions}
\label{sec:assumptions}

Throughout, $p:\B\to G$ will be a saturated, separable Fell bundle
over a locally compact groupoid $G$ as defined in \cite{muhwil:dm08}.
Thus $p:\B\to G$ is an upper semicontinuous Banach bundle over a
second countable locally compact groupoid $G$ such that its continuous
sections, vanishing at infinity on $G$, $\sa_{0}(G,\B)$, form a
separable Banach space with respect to the supremum norm.
Furthermore, all our groupoids are assumed to be second countable,
locally compact and Hausdorff.  When $G$ is a groupoid, it will be
assumed to have a Haar system $\sset{\lambda^{u}}_{u\in\go}$.  We will
write $A=\sa_{0}(\go,\B)$ for the \cs-algebra of $\B$ over $\go$.  We
then follow \cite{muhwil:dm08}*{\S1} to make the compactly supported
continuous sections $\sa_{c}(G,\B)$ into a $*$-algebra with
\cs-completion $\cs(G,\B)$.  When dealing with any sort of Banach
bundle $p:\B\to X$, we will use a roman font, $B(x)$, to indicated the
fibre over $x\in X$ together with its Banach space structure.  When we
have the need to work with a Fell bundle over a group --- in
particular, when we restrict a Fell bundle over a groupoid to a subgroup
--- we will, for the sake of consistency, treat the underlying group
as a groupoid as regards our conventions with modular functions.\footnote{The
  issue is that in the convolution algebra, the involution for
  groupoids has no modular function (since groupoids don't have
  modular functions until a quasi-invariant measure is picked).  The
  modular function then reappears in the integrated forms of
  representations.}  (See \cite{kmqw:nyjm10}*{\S1.5} for an
elaboration on this.)

\section{Inducing from $G(u)$}
\label{sec:step-i}

In this section we prove the following generalization of
\cite{ionwil:pams08}*{Theorem~5} to Fell bundle \cs-algebras.

\begin{thm}
  \label{thm:indirr}
  Suppose that $p:\B\to G$ is a separable Fell bundle over a locally
  compact groupoid with a Haar system.  Let $u\in\go$ and let
  $G(u):=\set{x\in G:r(x)=u=s(x)}$ be the stability group at $u$.
  Suppose that $L$ is an irreducible representation of $\cs(G(u),\B)$.
  Then $\Ind_{G(u)}^{G}L$ is an irreducible representation of
  $\cs(G,\B)$.
\end{thm}

Our proof of Theorem~\ref{thm:indirr} follows that of
\cite{ionwil:pams08}*{Theorem~5} very closely.  Because the
notation for the convolution of sections, inner products and actions
is virtually identical to the scalar-valued case, there are parts
where the proof can be used \emph{mutatis mutandis} from
\cite{ionwil:pams08}.  While this is one of the benefits of the Fell bundle
formalism, Theorem~\ref{thm:indirr} is a highly nontrivial generalization of
the scalar version, and we have tried to be careful below to point out
the places where we have had to adjust from working with scalar-valued
functions to sections of a nontrivial Banach bundle.  At the same
time, it seemed prudent to retain enough of the original argument from
the scalar case that the exposition remains readable.

\subsection{\boldmath Induced Representations of Fell Bundle \cs-Algebras}
\label{sec:induc-repr-fell}

We begin by recalling the construction of induced representations for
Fell bundles over groupoids from \cite{simwil:nyjm13}*{\S4.1}. Let
$q:\B\to G$ be a separable Fell bundle and assume that $H$ is a closed
subgroupoid of $G$. Let $q_H:\B|_H\to H$ be the Fell bundle obtained
by restriction to $H$. Then $G_{H^{(0)}}=s^{-1}(H^{(0)})$ is an
$(H^G,H)$-equivalence, where $H^G$ is the imprimitivity groupoid
$(G_{H^{(0)}}\ast_sG_{H^{(0)}})/H$. If $\sigma:H^G\to G$ is the
continuous map given by $\sigma([x,y])=xy^{-1}$, the pull-back Fell
bundle $\sigma^*q:\sigma^*\B\to H^G$ is the Fell bundle
$\sigma^*\B=\set{([x,y],b)\,:\,[x,y]\in H^G,b\in
  \B,\sigma([x,y])=q(b)}$ with bundle map
$\sigma^*q([x,y],b)=[x,y]$. Then $\E=q^{-1}(G_{H^{(0)}})$ is a
$\sigma^*\B-\B|_H$-equivalence with the left action of $\sigma^*\B$
given by $([x,y],b)\cdot e=be$ if $q(e)=yh$, the right action of
$\B|_H$ given by $e\cdot b=eb$, and the left and right inner products
on $\E\ast_s\E$ given by
\[
\lip{\sigma^*\B}<e,f>=([q(e),q(f)],ef^*)\quad\text{and}\quad 
\rip{\B|_H}<e,f>=e^*f.
\]
Therefore $\sa_{c}(G_{H^{(0)}},\E)$ is a pre-imprimitivity bimodule
with actions and inner products determined by
\begin{align*}
  F\cdot \varphi(z)&=\int_GF([z,y])\varphi(y)\,d\lambda_{s(z)}(y),\\
  \varphi\cdot g(z)&=\int_H\varphi(zh)g(h^{-1})\,d\alpha^{s(z)}(h),\\
  \Rip<\varphi,\psi>(h)&=\int_G\varphi(y)^*\psi(yh)\,d\lambda_{r(h)}(y),\\
  \Lip<\varphi,\psi>([x,y])&=\int_H
  \varphi(xh)\psi(yh)^*\,d\alpha^{s(x)}(h).
\end{align*}
The completion $X=X_H^G$ is a
$\cs(H^G,\sigma^*\B)-\cs(H,\B|_H)$-imprimitivity bimodule.

If $L$ is a representation of $\cs(H,\B|_H)$, then we write $\xind{L}$
for the representation of $\cs(H^G,\sigma^*{\B})$ induced via
$X$. Recall (see, for example, \cite{rw:morita}*{Proposition 2.66})
that $\xind{L}$ acts on the completion $\HH_{\Ind{L}}$ of $X\odot
\HH_L$ with respect to
\[
\bip(\varphi\tensor h|\psi\tensor
k)=\bip(L(\Rip<\psi,\varphi>)h|k)_{\HH_L}
\]
via
\[
(\xind{L})(F)(\varphi\tensor h)=F\cdot \varphi\tensor h.
\]
The induced representation of $\cs(G;\B)$ acts on $\HH_{\Ind{L}}$ by
\[
(\indgh L)(f)(\varphi\tensor h)=f\ast \varphi\tensor h,
\]
where $f\ast
\varphi(z)=\int_Gf(y)\varphi(y^{-1}z)\,d\lambda^{r(z)}(y)$ for $f\in
\Gamma_c(G,\B)$ and $\varphi\in \Gamma_c(G_{H^{(0)}},\E)$.

\subsection{The Proof of Theorem~\ref{thm:indirr}}
\label{sec:proof-theorem}

For the proof, we consider the case $H=G(u)$ for
some $u\in\go$. Then $H^{(0)}=\sset{u}$ and $G_{H^{(0)}}=G_u$.

Let $L$ be an irreducible
representation of $\cs(G(u),\B)$. Since $\X$ is a
$\cs(G(u)^G,\sigma^*\B)-\cs(G(u),\B|_{G(u)})$-imprimitivity bimodule,
\cite{rw:morita}*{Corollary 3.32} implies that $\xind L$ is an
irreducible representation of $\cs(G(u)^G,\sigma^*\B)$. To prove the
theorem, we just need to see that any 
$T$ in the commutant of $\Ind_{G(u)}^GL$ is a scalar multiple
of the identity. It will suffice to see that any such $T$ commutes
with $(\xind{L}) (F)$ for all $F\in
\Gamma_c(G(u)^G,\sigma^*\B)$. Hence, given such an $F$, we need to 
produce a net
$\sset{f_i}$ in $\Gamma_c(G,\B)$ such that
\[
(\Ind_{G(u)}^{G}L)(f_i)\to (\xind{L})(F)
\]
in the weak operator topology. We will arrange that this net is
uniformly bounded in the $\Vert \cdot\Vert_I$-norm on $\Gamma_c(G,\B)$
--- so that the net $\sset{(\Ind_{G(u)}^G{L})(f_i)}$ is uniformly
bounded in $B(\HH_{\Ind{L}})$. Then we just have to arrange that
\[
\bip((\Ind_{G(u)}^G{L})(f_i)(\varphi\tensor h)|\psi\tensor
k)\to\bip((\xind{L})(F)(\varphi\tensor h)|\psi\tensor k)
\]
for all $\varphi,\psi\in \Gamma_c(G_u,\E)$ and $h,k\in \HH_{\Ind{L}}$.

As in the proof of \cite{ionwil:pams08}*{Theorem 5}, the following
lemma is the essential ingredient in our proof.

\begin{lemma}\label{lemma6}
  Suppose that $F\in \Gamma_c(G(u)^G,\sigma^*\B)$. Then there is a
  compact set $C_F$ in $G$ such that for each compact set $K\subset
  G_u$ there is an $f_K\in \Gamma_c(G,\B)$ such that
  \begin{enumerate}
  \item $f_K(zy^{-1})=F([z,y])$ for all $(z,y)\in K\times K$,
  \item $\supp f_K\subset C_F$ and
  \item $\Vert f_K\Vert_I\le \Vert F\Vert_I+1$.
  \end{enumerate}
\end{lemma}

In order to prove lemma \ref{lemma6}, we need a version of
\cite{ionwil:pams08}*{Lemma 7} for semicontinuous functions.
\begin{lemma}\label{lemma7}
  Suppose that $f$ is a non-negative upper semicontinuous function on
  $G$ with compact support and that $K\subset G$ is a compact set such
  that
  \[
  \int_{K}f(x)\,d\lambda^{u}(x)\le M\quad\text{for all $u\in G^{(0)}$.}
  \]
  Then there is a neighborhood $V$ of $K$ such that
  \[
  \int_{V}f(x)\,d\lambda^{u}(x)\le M+1\quad\text{for all $u\in G^{(0)}$.}
  \]
\end{lemma}
\begin{proof}
  Let $K_{1}$ a compact neighborhood of $K$ and $\{V_{n}\}$ a
  countable fundamental system of neighborhoods of $K$ in $K_{1}$ such
  that $V_{n+1}\subset V_{n}$.  Assuming to the contrary that no $V$
  as prescribed in the lemma exists, then we can find a sequence
  $\{u_{n}\}\subset G^{(0)}$ such that
  \[
  \int_{V_{n}}f(x)\,d\lambda^{u_{n}}(x)>M+1.
  \]
  As in \cite{ionwil:pams08}, we can assume that
  $u_{n}\to u_{0}$. The dominated convergence theorem implies
  that
  \[
  \int_{V_{n}}f(x)\,d\lambda^{u_{0}}(x)\to\int_{K}f(x)\,d\lambda^{u_{0}}(x).
  \]
  In particular, there is an $n_{1}$ such that
  \[
  \int_{V_{n_{1}}}f(x)\,d\lambda^{u_{0}}(x)\le M+\frac{1}{2}.
  \]
  Let $W_{1}$ be an open set such that $K\subset
  W_{1}\subset\overline{W}_{1}\subset V_{n_{1}}$.  Then
  $1_{\overline{W}_{1}}f$ is upper semicontinuous and
  \[
  \int_{\overline{W}_{1}}f(x)\,d\lambda^{u_{0}}\le M+\frac{1}{2}.
  \]
  Let $0<\varepsilon<\frac{1}{2}$. Using \cite{muhwil:dm08}*{Lemma
    3.4} we can find $g\in C_{c}^{+}(G)$ such that $f(x)\le g(x)$ for
  all $x\in\overline{W}_{1}$ and
  \[
  \int_{\overline{W}_{1}}f(x)\,d\lambda^{u_{0}}(x)
  \le\int_{G}g(x)\,d\lambda^{u_{0}}(x)<M+\frac{1}{2}+\varepsilon.  
  \]
  Let $W$ be open such that $K\subset W\subset\overline{W}\subset
  W_{1}$ and let $f_{0}\in C_{c}^{+}(G)$ be such that
  $f_{0}|_{\overline{W}}=g$, $f_{0}\le g$, and
  $\operatorname{supp}f_{0}\subset W_{1}$. Then
  \[
  \int_{G}f_{0}(x)\,d\lambda^{u_{0}}(x)<M+\frac{1}{2}+\varepsilon.
  \]
  However, since $\{\lambda^{u}\}$ is a Haar system,
  \[
  \int_{G}f_{0}(x)\,d\lambda^{u_{m}}(x)\to\int_{G}f_{0}(x)\,d\lambda^{u_{0}}(x).
  \]
  Therefore, for large $n$, we have that
  $\int_{G}f_{0}(x)\,d\lambda^{u_{m}}(x)<M+1$. However, for large $n$,
  we have $V_{n}\subset W$ and therefore
  \begin{align*}
    \int_{G}f_{0}(x)\,d\lambda^{u_{n}}(x) &
    \ge\int_{V_{n}}f_{0}(x)\,d\lambda^{u_{n}}(x)=\int_{V_{n}}g(x)\,d\lambda^{u_{n}}(x)\\ 
    & \ge\int_{V_{n}}f(x)\,d\lambda^{u_{n}}(x)>M+1.
  \end{align*}
  This leads to a contradiction and the proof is complete.
\end{proof}
\begin{proof}[Proof of Lemma \ref{lemma6}]

  The map $(z,y)\mapsto zy^{-1}$ is continuous on $G_u\times G_u$ and
  factors through the orbit map $\pi:G_u\times G_u\to
  G(u)^G$. Moreover, the continuous map $\sigma:G(u)^G\to G$ defined
  via $\sigma([z,y])=zy^{-1}$ is injective. We let $C_F$ be a compact
  neighborhood of $\sigma(\supp F)$.

  Fix a compact set $K\subset G_u$. Then the restriction of $\sigma$
  to the compact set $\pi(K\times K)$ is a homeomorphism. Using the
  vector-valued Tietze Extension Theorem
  (\cite{muhwil:dm08}*{Proposition A.5}) we can find a section
  $\tilde{f}_K\in \Gamma_c(G,\B)$ such that $\supp \tilde{f}_K\subset
  C_F$ and such that $\tilde{f}_K(zy^{-1})=F([z,y])$ for all $(z,y)\in
  K\times K$.

  Let $K_G:=\sigma\bigl(\pi(K\times K)\bigr)\subset G$. If
  \[
  \int_{K_G}\Vert \tilde{f}_K(y)\Vert\,d\lambda^w(y)\ne 0,
  \]
  then $K_G\bigcap G^w\ne \emptyset$.  Therefore there is $z\in K$
  such that $r(z)=w$. Then by left invariance
  \begin{align*}
    \int_{K_G}\Vert \tilde{f}_K(y)\Vert \,d\lambda^w(y)&=\int_G
    \charfcn{K_G}(zy)\Vert \tilde{f}_K(zy)\Vert \,d\lambda^w(y)\\
    &=\int_G\charfcn{K_G}(zy^{-1})\Vert \tilde{f}_K(zy^{-1})\Vert
    \,d\lambda_u(y)\\
    &=\int_G\charfcn{K_G}(zy^{-1})\Vert F([z,y])\Vert \,d\lambda_u(y)\\
    &\le \Vert F\Vert_I.
  \end{align*}
  Similarly, if
  \[
  \int_{K_G}\Vert \tilde{f}_K(y^{-1})\Vert \,d\lambda^w(y)\ne 0,
  \]
  then as before there is a $z\in K$ such that $r(z)=w$ and
  \begin{align*}
    \int_{K_G}\Vert \tilde{f}_K(y^{-1})\Vert \,d\lambda^w(y)&=
    \int_G\charfcn{K_G}(zy)\Vert \tilde{f}_K(y^{-1}z^{-1})\Vert
    \,d\lambda^w(y)\\
    &=\int_G\charfcn{K_G}(zy^{-1})\Vert \tilde{f}_K(yz^{-1})\Vert
    \,d\lambda_u(y)\\
    \intertext{which, since $K_{G}^{-1}=K_{G}$, is} &= \int_G
    \charfcn{K_G}(yz^{-1})\Vert \tilde{f}_K(yz^{-1})\Vert
    \,d\lambda_u(y)\\
    &\le \int_G\Vert F([y,z])\Vert \,d\lambda_u(y)\le \Vert F\Vert_I.
  \end{align*}
  Using Lemma \ref{lemma7}, we can find a neighborhood $V$ of $K_G$
  contained in $C_F$ such that both
  \[
  \int_V\Vert \tilde{f}_K(y)\Vert
  \,d\lambda^w(y)\quad\text{and}\quad\int_{V}\Vert
  \tilde{f}_K(y^{-1})\Vert \,d\lambda^w(y)
  \]
  are bounded by $\Vert F\Vert_I+1$ for all $w\in \go$. Since $K_G$ is
  symmetric we can assume that $V=V^{-1}$ as well. Using 
  the vector-valued Tietze extension theorem, we let $f_K$ be any element
  of $\Gamma_c(G,\B)$ such that $f_K=\tilde{f}_K$ on $K_G$, $\supp
  f_K\subset V$ and $\Vert f_K(x)\Vert \le \Vert \tilde{f}_K(x)\Vert$
  everywhere. Then $f_K$ satisfies the conclusion of the Lemma.
\end{proof}
\begin{proof}[Proof of Theorem \ref{thm:indirr}]
  For each $K\subset G_{u}$, let $f_{K}$ be as in Lemma~\ref{lemma7}.
  Then $\set{f_{K}}$ and $\set{(\indgug L)(f_{K})}$ are nets indexed
  by increasing $K$. Notice that
  \begin{multline}\label{eq:2_1}
    \bip((\indgug L)(f_{K})(\phi\tensor_{G(u)}h) |
    \psi\tensor_{G(u)}k) -{}\\ \bip((\xind L)(F) (\phi\tensor_{G(u)}h)
    |
    \psi\tensor_{G(u)}k) \\
    = \bip( L\bigl(\Rip<\psi,f_{K}*\phi-F\cdot \phi>\bigr)h|k)
  \end{multline}
  Furthermore, using the invariance of the Haar system on $G$, we can
  compute as follows:
  \begin{equation}
    \label{eq:2_2}
    \begin{split}
      \Rip<\psi,f_{K}*\phi>(s)&= \int_{G}{\psi(x)}^*
      f_{K}*\phi(xs) \, d\lambda_{u}(x) \\
      &= \int_{G}\int_{G} {\psi(x)}^* f_{K}(xz^{-1})\phi(zs)
      \,d\lambda_{u}(z) \,d\lambda_{u}(x),
    \end{split}
  \end{equation}
  while on the other hand,
  \begin{equation}
    \label{eq:2_3}
    \begin{split}
      \Rip<\psi,F\cdot \phi>(s)&= \int_{G}{\psi(x)}^* F\cdot
      \phi(xs) \,d\lambda_{u}(x) \\
      &=\int_{G}\int_{G} {\psi(x)}^* F\bigl([xs,z]\bigr) \phi(z)
      \,d\lambda_{u}(z) \,d\lambda_{u}(x) \\
      &= \int_{G}\int_{G} {\psi(x)}^* F\bigl([x,zs^{-1}]\bigr) \phi(z)
      \,d\lambda_{u}(z) \,d\lambda_{u}(x) \\
      &= \int_{G}\int_{G} {\psi(x)}^* F\bigl([x,z]\bigr) \phi(zs)
      \,d\lambda_{u}(z) \,d\lambda_{u}(x).
    \end{split}
  \end{equation}

  Notice that $\supp\Rip<\psi,\phi>\subset (\supp \psi)(\supp\phi)$.
  Since $\supp f_{K}\subset C_{F}$ for all $K$, we have
  \begin{equation*}
    \supp f_{K}*\phi\subset (\supp f_{K})(\supp \phi)\subset
    C_{F}(\supp\phi).
  \end{equation*}
  Therefore if \eqref{eq:2_2} does not vanish, then we must have $s\in
  (\supp\psi)C_{F}(\supp \phi)$.  Therefore there is a compact set
  $K_{0}$ --- which does \emph{not} depend on $K$ --- such that both
  \eqref{eq:2_2} and \eqref{eq:2_3} vanish if $s\notin K_{0}$.  Thus
  if $s\in K_{0}$ and if $K \supset (\supp\psi)\cup (\supp
  \phi)K_{0}^{-1}$, then the integrand in \eqref{eq:2_2} and
  \eqref{eq:2_3} are both zero or we must have $(x,z)\in K\times K$.
  Therefore we can replace $f_{K}(xz^{-1})$ by $F\bigl([x,z]\bigr)$,
  and then $f_{K}*\phi-F\cdot\phi$ is the zero section whenever $K$
  contains $(\supp\psi)\cup (\supp \phi)K_{0}^{-1}$.  Therefore the
  left-hand side of \eqref{eq:2_1} is eventually zero, and the theorem
  follows.
\end{proof}

The proof of the following theorem is similar to that
of \cite{ionwil:pams08}*{Theorem 4} and we will omit it. The result will
be useful in Section~\ref{sec:group-case}.

\begin{thm}[Induction in stages]\label{thm:indstages} Suppose that
  $q:\B\to G$ is a separable Fell bundle over a second countable
  locally compact 
  Hausdorff groupoid $G$ and that $H$ and $K$ are closed subgroupoids
  of $G$ with $H\subset K$. Assume that $H$, $K$, and $G$ have Haar
  systems. If $L$ is a representation of $\cs(H,\B)$, then
  \[
  \Ind_H^G L\;\text{and}\;\Ind_K^G\bigl(\Ind_H^K L\bigr)
  \]
  are equivalent representations of $\cs(G,\B)$.
\end{thm}

\section{The Strong Effros-Hahn 
Induction Property for Fell Bundles}
\label{sec:red-group-case}

Recall that, for a Fell bundle $q:\B\to G$ over a groupoid $G$,
$A=\sa_{0}(\go,B)$ is a $C^*$-algebra that we call the $C^*$-algebra of
$\B$ over $\go$. Then $A$ is a $C_0(\go)$-algebra and we let
$\sigma_A:\Prim A\to\go$ be the associated structure map (see, for
example, \cite{wil:crossed}*{\S C.1}). If $u\in \go$, let $p_u:A\to
A(u)$ be the quotient map with kernel $I(u)$. Note that if $P\in \Prim
A$, then $u=\sigma_{A}(P)$ is the \emph{unique} $u\in\go$  
such that $P\supset I(u)$.  Moreover, $\Prim A$ is naturally
identified with the disjoint union of the $\Prim A(u)$
\cite{wil:crossed}*{Proposition C.5}. 

Recall that $G$ acts on $\Prim A$ via the Rieffel correspondence
$h_x:\Prim A(s(x))\to \Prim A(r(x))$ (see
\cite{ionwil:hjm11}*{\S2}).  Thus, if $P\in\Prim A$ and if $x\in G$ is
such
that $\sigma_A(P)=s(x)$, then $x\cdot P=h_x(P)$.
Since the $G$-action $h_{x}$
also maps $I(s(x))$ to $I(r(x))$ (see \cite{rw:morita}*{Proposition~3.24}), the stability
group $G_{P}$ of $P$ is a subgroup of the stability group $G(u)$,
where $u=\sigma_A(P)$.  

Suppose now that $q:\B\to G$ is a separable Fell bundle over a
\emph{group} $G$.\footnote{Notice that in this case the underlying
  Banach bundle is continuous 
\cite{bmz:pems13}*{Lemma~3.30}.}  If $L:\cs(G,\B)\to B(\WW)$ is a
representation, then there is a strictly continuous, nondegenerate
$*$-homomorphism $\pi:\B\to B(\WW)$ such that
\begin{equation}
  \label{eq:1}
  L(f) =\int_{G} \pi\bigl(f(s)\bigr) \Delta_{G}(s)^{-\half}\,d\mu_{G}(s).
\end{equation}
This is a consequence of \cite{kmqw:nyjm10}*{Lemma~1.3}.  (The
appearance of the modular
function in \eqref{eq:1} is a consequence of our convention of 
treating  $G$ as a groupoid: see
\cite{kmqw:nyjm10}*{Remark~1.5}.)  
Notice that $\pi\restr A$ is a representation of $A$ on $\WW$.
Let $I$ be an ideal in $A = B(e)$.  Then we say that $L$ or $\pi$ has
kernel $I$ if $\ker (\pi\restr A)=I$. 

In
analogy with \cite{echwil:tams08}*{Definition~1.1}, we make the
following definition.

\begin{definition}
  \label{def-ehi}
  We say that a Fell bundle $q:\B\to G$ over a groupoid $G$ with
  \cs-algebra $A$ over $\go$ satisfies
  the \emph{strong Effros-Hahn induction property} (strong-EHI) if
  given $P\in \Prim A$ and an irreducible representation $L$ of
  $\cs(G_{P},\B)$ with kernel $P$, then $\Ind_{G_P}^{G}L$ is irreducible.
\end{definition}


\section{The strong-EFI for homogeneous representations}
\label{sec:group-case}

In this section we prove our main result, which asserts that the strong-EHI
property holds under the additional assumption that the restriction of
the representation to $A$ is homogeneous.

\begin{thm}
  \label{thm-main-1.7}
  Let $q:\B\to G$ be a saturated, separable Fell bundle over a locally
  compact groupoid $G$ and suppose that $P\in \Prim A$ where
  $A=\sa_{0}(\go,\B)$ is the associated \cs-algebra over $\go$.  Suppose that
  $L$ is an irreducible representation of $\cs(G_{P},\B\restr{G_{P}})$
  which is the integrated form of $\pi:\B\restr{G_{P}}\to B(\WW)$ with
  $\pi\restr A$ homogeneous with kernel $P$.  Then $\Ind_{G_{P}}^{G}
  L$ is an irreducible representation of $\cs(G,\B)$.
\end{thm}

\begin{remark}\label{rem:reducing}
  Let $L$ be a representation of $\cs(G_{P},\B)$ with kernel $P$. Then
  Theorem~\ref{thm:indstages} implies that $\Ind_{G_P}^G{L}$ and
  $\Ind_{G(u)}^G\bigl(\Ind_{G_P}^{G(u)}{L}\bigr)$ are equivalent
  representations, where $u=\sigma_A(P)$.  Together with
  Theorem~\ref{thm:indirr}, this shows that the irreducibility of
  $\Ind_{G_{P}}^{G(u)}L$ implies that of $\Ind_{G_{P}}^{G}L$.  Hence,
  in order to show that a Fell bundle satisfies strong-EHI, it will
  suffice to consider the case where $G$ is a group.
\end{remark}

In view of Remark~\ref{rem:reducing}, we 
will assume for the remainder of this
section that $G$ is a \emph{group}.  Hence, we fix a Fell bundle $q:\B\to G$ 
over a group $G$ and let $A=B(e)$ be the corresponding \cs-algebra.

To start, let $H$ be any closed subgroup of $G$.  Fix Haar
measures $\mug$ and $\muh$ on $G$ and $H$, respectively.  Let
$\rho:G\to (0,\infty)$ be a continuous function satisfying
\begin{equation}
  \label{eq:2}
  \rho(xh)=\frac{\Delta_{H}(h)}{\Delta_{G}(h)}\rho(x)\quad\text{for
    all $x\in G$ and $h\in H$.}
\end{equation}
For convenience later, we normalize $\rho$ so that $\rho(e)=1$.  Then
there is a quasi-invariant measure $\mugmh$ on $G/H$ (see, for example
\cite{rw:morita}*{Lemma~C.2}) such that
\begin{equation}
  \label{eq:3}
  \int_{G} f(x)\rho(x)\,d\mug(x)=\int_{G/H}\int_{H}
  f(xh)\,d\muh(h)\,d\mugmh(\dot x)\quad\text{for all $f\in C_{c}(G)$.}
\end{equation}
In fact, $\mugmh$ is quasi-invariant when viewed as a measure on the
unit space of the transformation groupoid $G\times G/H$ (which we
identify with $G/H$). Recall that two elements $(x,yH)$ and $(z,wH)$ in
$G\times G/H$ are composable provided that $wH=x^{-1}yH$ and
$(x,yH)(z,x^{-1}yH)=(xz,yH)$. The inverse of $(x,yH)$ is
$(x^{-1},x^{-1}yH)$. It follows that $s(x,yH)= x^{-1}yH$ and
$r(x,yH)=yH$. The Haar system on the transformation groupoid is given
by $\lambda=\{\mu\times \epsilon_{yH}\}_{yH\in G/H}$.

\begin{lemma}
  \label{lem-modular}
  The modular function on the transformation groupoid $G\times G/H$
  with respect to the quasi-invariant measure $\mugmh$ defined in
  \eqref{eq:3} is
  \begin{equation}
    \label{eq:7}
    \delta(x,yH)=\Delta_{G}(x)\frac{\rho(y)}{\rho(x^{-1}y)}.
  \end{equation}
\end{lemma}
\begin{proof}
  Let $F\in C_c(G\times G/H)$. Then
  \begin{align*}
    \lambda(F)&=\int_{G/H}\int_{G} F(x,yH)\,d\mug(x)\,d\mugmh(\dot y) \\
    &=\int_{G/H}\int_{G} F(x^{-1},yH)\Delta_{G}(x^{-1})
    \,d\mug(x)\,d\mugmh(\dot y) \\
    \intertext{which, using Fubini and \cite{rw:morita}*{Lemma~C.2},
      is} &=
    \int_{G/H}\int_{G/H}\int_HF((xh)^{-1},(xh)^{-1}y)\Delta_G((xh)^{-1})
    \\
    &\hskip 2in
    \frac{\rho((xh)^{-1}y)}{\rho(y)}\frac{1}{\rho(xh)}d\nu(h)d\mugmh(\dot
    y)d\mugmh(\dot x)\\
    &= \int_{G/H}\int_{G}
    F(x^{-1},x^{-1}yH)\Delta_{G}(x^{-1})\frac{\rho(x^{-1}y)}{\rho(y)}
    \,d\mug(x)\,d\mugmh(\dot y). 
  \end{align*}
  It follows that
  \begin{equation*}
    \delta((x,yH)^{-1})=\delta(x^{-1},x^{-1}yH)= \Delta_{G}(x^{-1})
    \frac{\rho(x^{-1}y)}{\rho(y)}. 
  \end{equation*}
  The conclusion follows.
\end{proof}

We now fix a representation $L$ of $\cs(H,\B)$ on $\WW$, and assume
that it is the integrated form of $\pi:\B\restr H\to B(\WW)$ as in
\eqref{eq:1}.  Note that the map $[x,y]\mapsto (xy^{-1},xH)$
allows us to identify the imprimitivity groupoid $H^{G}=(G\times G)/H$
with the transformation groupoid $G\times G/H$.
Then we can complete $\sa_{c}(G,\B)$ to an imprimitivity bimodule $\X$
between $\cs(G\times G/H, \pr_{1}^{*}\B)$ and $\cs(H,\B)$, where
$\pr_{1}:G\times G/H$ is the projection onto the $G$-factor. Hence
$\pr_{1}^{*}\B=\set{(y,xH,b):b\in B(y)}$. Then the induced
representations $\xind L$ of $\cs(G\times G/H,\pr_{1}^{*}\B)$ and
$\indgh L$ of $\cs(G,\B)$ both act on the completion of
\begin{equation*}
  \sa_{c}(G,\B)\odot \WW
\end{equation*}
with respect to the pre-inner product
\begin{align}
  \bip(f\tensor& \xi| g\tensor \eta) = \bip(L(\rip\star<g,f>)\xi|\eta)
  \notag \\
  & = \int_{H} \bip( \pi\bigl(\rip\star<g,f>(h)\bigr)\xi| \eta)
  \Delta_{H}(h)^{-\half} \,d\muh(h) \notag \\
  & =\int_{H}\int_{G}
  \bip(\pi\bigl(g(x^{-1})^{*}f(x^{-1}h)\bigr)\xi|\eta)
  \Delta_{H}(h)^{-\half} \,d\mug(x)\,d\muh(h) \notag \\
  & =\int_{H}\int_{G} \bip(\pi\bigl(g(x)^{*}f(xh)\bigr)\xi|\eta)
  \Delta_{G}(x)^{-1}\Delta_{H}(h)^{-\half} \,d\mug(x)\,d\muh(h) \notag \\
  & =\int_{H}\int_{G/H}\int_{H}
  \bip(\pi\bigl(g(xr)^{*}f(xrh)\bigr)\xi|\eta)
  \Delta_{G}(xr)^{-1}\Delta_{H}(h)^{-\half} \rho(xr)^{-1} \notag \\
  &\hskip3.75in \,d\muh(r)\,d\mugmh(\dot x)
  \,d\muh(h) \notag \\
  &= \int_{G/H} \bip(f\tensor \xi|g\tensor \eta)_{xH} \,d\mugmh (\dot
  x),
  \label{eq:4}
\end{align}
where we define
\begin{multline}
  \label{eq:5}
  \bip(f\tensor \xi|g\tensor \eta)_{xH}\\:= \int_{H}\int_{H}
  \bip(\pi\bigl(g(xr)^{*}f(xh)\bigr)\xi|\eta)
  \Delta_{G}(xr)^{-1}\Delta_{H}(r^{-1}h)^{-\half} \rho(xr)^{-1}
  \,d\muh(h)\,d\muh(r).
\end{multline}

As we described in Section~\ref{sec:induc-repr-fell}, the induced
representations are given by 
\begin{equation*}
  (\xind L)(F)(g\tensor\eta)=F\cdot g\tensor \eta\quad\text{and} \quad
  (\indgh L)(f)(g\tensor \eta)=f*g\tensor\eta,
\end{equation*}
where
\begin{gather*}
  F\cdot g(x)=\int_{G} F(y,xH)g(y^{-1}x) \,d\mug(y) 
\quad\text{and}\quad
f*g(x)=\int_{G} f(y)g(y^{-1}x)\,d\mug(y).
\end{gather*}

We will need the following technical lemma.
\begin{lemma}
  \label{lem-tech-pos}
  For each $x\in G$, the sesquilinear form $\ip(\cdot|\cdot)_{xH}$ is
  positive semi-definite on $\sa_{c}(G,\B)\odot \WW$ and therefore a
  pre-inner product.
\end{lemma}
\begin{proof}
It is not hard to check that $f\mapsto \ip(f\tensor\xi|f\tensor\xi)$
is continuous in the inductive limit topology.  It follows from
\cite{muhwil:nyjm08}*{Lemma~6.1} that if $\sset{a_{\kappa}}$ is an
approximate unit for $A$ and $f\in\sa_{c}(G,\B)$, then $a_{\kappa}\cdot
f\to f$ in the inductive limit topology.  Hence
\begin{equation}\label{eq:13}
  \Bip(\sum_{i=1}^{n} f_{i}\tensor\xi_{i} | \sum_{j=1}^{n}
  f_{j}\tensor \xi_{j}) = \lim_{\kappa} \Bip(\sum_{i=1}^{n}
  a_{\kappa}\cdot f_{i}\tensor\xi_{i} | \sum_{j=1}^{n} 
  a_{\kappa}\cdot f_{j}\tensor \xi_{j}) .
\end{equation}
Since $B(x)$ is an $A\sme A$-\ib, sums of the form $\sum_{r}
c_{r}c_{r}^{*}$, with each $c_{r}\in B(x)$, are dense in $A^{+}$ (see
\cite{muhwil:nyjm08}*{Lemma~6.3}).  Since each $a_{\kappa}\ge0$ in
$A$, this implies that we can replace $a_{\kappa}$ by a sum $\sum_{r}
c_{r}^{\kappa}(c_{r}^{\kappa})^{*}$ with $c_{r}^{\kappa}\in B(x)$ and
still have \eqref{eq:13} hold.

But then since $\rho(e)=1$ we have

  \begin{equation*}
    \label{eq:12}
    \Delta_{G}(xr)^{-1}\rho(xr)^{-1}\Delta_{H}(r^{-1}h)^{-\half} =
    \Delta_{G}(x)^{-1} \Delta_{H}(hr)^{-\half}\rho(x)^{-1},
  \end{equation*}
  and
  \begin{equation*}
    \label{eq:18}
    \bip( \sum_{i} a_{\kappa}\cdot f_{i}\tensor \xi_{i}  | \sum_{j}
    a_{\kappa}\cdot f_{j}\tensor\xi_{j} )_{xH} = \sum_{r}
    \ip(\eta_{r}^{\kappa}|\eta_{r}^{\kappa})
  \end{equation*}
  where
  \begin{equation*}
    \label{eq:19}
    \eta_{r}^{\kappa} = \sum_{i}
    \int_{H}\Delta_{G}(x)^{-\half}\Delta_{H}(h)^{-\half}\pi((c_{r}^{\kappa})^*f_{i}(xh))\xi_{i} 
    \,d\muh(h).
  \end{equation*}
  Since $\sum_{k} \ip(\eta_{k}|\eta_{k})$ is nonnegative, it follows
  from \eqref{eq:13}
  that our form is a pre-inner product as claimed.
\end{proof}


In view of Lemma~\ref{lem-tech-pos}, the (Hausdorff) completion of
$\sa_{c}(G,\B)\odot\WW$ with respect to $\ip(\cdot|\cdot)_{xH}$ is a
Hilbert space $\H(xH)$.  We denote the image of $f\tensor\xi$ in
$\H(xH)$ by $f\tensor_{xH}\xi$.
Note that the class of $f\tensor_{xH}\xi$ depends only on
$f\restr{xH}$.  Furthermore, in view of the Tietze Extension Theorem
for upper semicontinuous Banach bundles
\cite{muhwil:dm08}*{Proposition~A.5}, every $f_{0}\in \sa_{c}(xH,\B)$
is the restriction of some $f\in \sa_{c}(G,\B)$.
Using \cite{wil:crossed}*{Proposition~F.8}, we can form a Borel
Hilbert bundle $G/H*\VV$, and identify the space of both induced
representations with the direct integral $L^{2}(G/H*\VV,\mugmh)$. We
will write $f\tensor\xi$ for the element in $L^{2}(G/H*\VV,\mugmh)$
given by $f\tensor\xi(xH)=f\tensor_{xH}\xi$.

If $b\in B(y)\subset \B$, then we define an operator $\iota_{\B}(b)$
on the vector space $\sa_{c}(G,\B)$ by
\begin{equation}
  \label{eq:6}
  \iota_{\B}(b)f(x)=\Delta_{G}(y)^{\half}bf(y^{-1}x).
\end{equation}
By \cite{kmqw:nyjm10}*{Lemma~1.2 and Remark~1.5}, $\iota_{\B}$ extends
to a homomorphism, also denoted $\iota_{\B}$, of $\B$ into
$M(\cs(G,\B))$.  It follows that $\indgh L$ is the integrated form of
$\Pihat:\B\to B(L^{2}(G/H*\VV,\mugmh))$ given
by
\begin{equation}\label{pihat}
  \Pihat(b)(f\tensor \xi) = \iota_{\B}(b)f\tensor \xi.
\end{equation}

Note that $\Pihat$ is not decomposible; that is, it does not operate
fibrewise on $L^{2}(G/H*\VV,\mugmh)$.  Nevertheless, it does play nice
with the fibres and is related to a Borel $*$-functor for $\pr_{1}\B$
(see \cite{muhwil:dm08}*{Definition~4.5}).

\begin{lemma}
  \label{lem-needed}\emergencystretch=50pt
  If $b\in B(y)\subset \B$, then we get an operator $\pihat(y,xH,b)\in
  B(\H(y^{-1}xH),\H(xH))$ of norm at most $\|b\|$ given by
  \begin{multline*}
    \pihat(y,xH,b)(f\tensor_{y^{-1}xH}
    \xi)=\pihat_{0}(y,xH,b)(f)\tensor_{xH}\xi, \quad\text{where}\\
    \pihat_{0}(y,xH,b)(f)(xh)= \delta(y,xH)^{\half}bf(y^{-1}xh).
  \end{multline*}
  (Recall that $\delta$ is the modular function on $G\times G/H$ for
  the quasi-invariant measure $\mugmh$ --- see
  Lemma~\ref{lem-modular}.)
\end{lemma}

\begin{proof}
  Using $\delta(y,xH)=\Delta_{G}(y)\frac{\rho(x)}{\rho(y^{-1}x)}$ as
  well as the observation that $\frac{\rho(x)}{\rho(y^{-1}x)} =
  \frac{\rho(xh)}{\rho(y^{-1}xh)}$ for any $h\in H$, we first check
  that
  \begin{equation*}
    \bip(\pihat(y,xH,b)(f\tensor\xi)|g\tensor \eta)_{xH}=  \bip(f\tensor
    \xi| \pihat(y^{-1},y^{-1}xH,b^{*})(g\tensor \eta))_{y^{-1}xH}.
  \end{equation*}
  However, in $\tilde A$, we have $\|b\|^{2}1_{\tilde{A}}-b^{*}b\ge0$.
  Hence there is a $c\in \tilde A$ such that
  $\|b\|^{2}1_{\tilde{A}}-b^{*}b=c^{*}c$.  Thus, extending
  $\pihat(e,xH,a)$ to all $a\in \tilde A$ in the usual way, we have
  \begin{align*}
    \|b\|^{2}\|f&\tensor\xi\|_{y^{-1}xH}^{2}-\|\pihat(y,xH,b)(f\tensor\xi)\|_{xH}^{2}\\
    &= \|b\|^{2}\bip(f\tensor\xi|f\tensor\xi)_{y^{-1}xH} -
    \bip(\pihat(y,xH,b)(f\tensor\xi)| {\pihat(y,xH,b)(f\tensor\xi)})_{xH} \\
    &= \|b\|^{2}\bip(f\tensor\xi|f\tensor\xi)_{y^{-1}xH} -
    \bip(f\tensor \xi|
    {\pihat(e,y^{-1}xH,b^{*}b)(f\tensor \xi})_{y^{-1}xH} \\
    &= \bip(\pihat(e, y^{-1}xH,c)(f\tensor\xi)|
    {\pihat(e,y^{-1}xH,c)(f\tensor\xi)})_{y^{-1}xH}
    \\
    &\ge0.
  \end{align*}
  Thus $\pihat(y,xH,b)$ is a well-defined operator from $\H(y^{-1}xH)$
  to $\H(xH)$ for any $x$, and has norm at most $\|b\|$.
\end{proof}

Now it is not hard to see that $\pihat$ is a Borel $*$-functor on
$\pr_{1}\B$, and more significantly, $\xind L$ is the integrated form
of $\pihat$ --- see \cite{muhwil:dm08}*{Proposition~4.10}.

At this point, we want to fix a primitive ideal $P\in \Prim A$ and
replace $H$ by $G_{P}$, where $G_{P}$ is the stability group for the
$G$ action on $\Prim A$ induced using the Rieffel correspondence: see
\cite{ionwil:hjm11}*{\S2} and Section \ref{sec:red-group-case} for
details.  In particular, if $x\in G$, then
\begin{equation}
  \label{eq:10}
  B(x) \cdot P= (x\cdot P) \cdot B(x).
\end{equation}

We let $\rho_{xG_{P}}$ be the representation of $A$ in $\H(xG_{P})$
given by $a\mapsto \pihat(e,xG_{P},a)$.  While \emph{a priori} the
$\rho_{xG_{P}}$ do not seem related to $\pi\restr A$, they in fact are
intimately so.

\begin{lemma}
  \label{lem-equiv-pi}
  The representation $\rho_{eG_P}$ defined above is equivalent to
  $\pi\restr A$.
\end{lemma}
\begin{proof}
  We claim that the map $U:\H(G_{P})\to \WW$ defined
  via
  \begin{equation*}
    U(f\tensor_{eG_{P}}\xi)=\int_{eG_P}\Delta_{H}(h)^{-\half}\pi(f(h))\xi\,d\muh(h),
  \end{equation*} is a unitary
  that intertwines $\rho_{eG_P}$ with $\pi\restr A$.

  Let $f\tensor \xi$ and $g\tensor \eta$ be two elements in
  $\Gamma_c(G,\B)\odot \WW$. Then
  \begin{align*}
    &\bigl(U(f\tensor_{eG_{P}} \xi)\bigm | U(g\tensor_{eG_{P}}\eta) \bigr) \\
    &=\Bigl(\int_{G_{P}}\Delta_H(h)^{-\frac{1}{2}}\pi(f(h))\xi
    \,d\nu(h)\Bigm
    |\int_{G_{P}}\Delta_H(r)^{-\frac{1}{2}}\pi(g(r))\eta
    \,d\nu(r)\Bigr)\\
    &=\int_{G_{P}}\int_{G_{P}}\bigl(\pi(f(h))\xi\bigm |
    \pi(g(r))\eta\bigl)\Delta_H(h)^{-\frac{1}{2}}\Delta_H(r)^{-\frac{1}{2}}
    \,d\nu(h)\,d\nu(r)\\
    \intertext{which, using the fact that $\pi\restr \B$ is a
      $\ast$-homomorphism and Equation \eqref{eq:12}, equals}
    &=\int_{G_{P}}\int_{G_{P}}\bigl(\pi(g(r)^\ast f(h))\xi\bigm |
    \eta\bigr)\Delta_G(r)^{-1}\rho(r)^{-1}\Delta_H(r^{-1}h)\,d\nu(h)\,d\nu(r)\\
    &=\bigl(f\tensor_{eG_{P}}\xi\bigm
    |g\tensor_{eG_{P}}\eta\bigr)_{eG_{P}}.
  \end{align*}
  The surjectivity of $U$ follows from the fact that there are enough
  sections and the fact that $\pi$ is nondegenerate (see
  \cite{simwil:ijm13}*{Section 3}). Thus $U$ is indeed a unitary that
  clearly intertwines $\rho_{eG_P}$ with $\pi\restr A$.
\end{proof}

Since each $B(x)$ is an $A\sme A$-\ib, we can form the induced
representation $\ibind{B(x)}( \rho_{eG_{P}})$ of $A$ on
$B(x)\tensor_{A}\H(eG_{P})$.

\begin{lemma}
  \label{lem-key-induced}
  For each $x\in G$, $\ibind{B(x)}(\rho_{eG_{P}})$ is equivalent to
  $\rho_{xG_P}$.
\end{lemma}
\begin{proof}
  Define $U:B(x)\odot \H(eG_{P})\to \H(xG_{P})$ by
  \begin{equation*}
    U(c \tensor
    f\tensor_{eG_{P}}\xi)=
    \Delta_{G}(x)^{\half}\rho(x)^{\half}(f_{c}\tensor_{xG_{P}}\xi),  
  \end{equation*}
  where $f_{c}$ is any section in $\sa_{c}(G,\B)$ satisfying
  \begin{equation*}
    f_{c}(xh)=cf(h).
  \end{equation*}
  To check that $U$ is isometric one can proceed as follows:
  \begin{align*}
    \bip(U(c\tensor {}&f\tensor_{eG_{P}}\xi) | {U(b\tensor
      g\tensor_{eG_{P}}\eta)})_{xG_{P}} =\Delta_G(x)\rho(x)\bip(f_c
    \tensor_{xG_{P}}\xi|
    g_b\tensor_{xG_{P}}\eta)_{xG_{P}}\\
    &=\int_{G_{P}}\int_{G_{P}}\bip(\pi(g_b(xr)^*f_c(xh))\xi
    |\eta)\Delta_G(x)\Delta_G(xr)^{-1}
    \Delta_H(r^{-1}h)^{-\half}\\
    &\hskip3in\rho(x)\rho(xr)^{-1}d\nu(h)d\nu(r)\\
    &=\int_{G_{P}}\int_{G_{P}}\bip(\pi(g(r)^*b^*cf(h))\xi|\eta)
    \Delta_G(r)^{-1}\Delta_H(r^{-1}h)^{-\half}\rho(r)^{-1}d\nu(h)d\nu(r)\\
    &=\int_{G_{P}}\int_{G_{P}}
    \bip(\pi(g(r)^*\hat{\pi}_0(e,eG_{P},b^*c)(f)(h))\xi|\eta)
    \Delta_G(r)^{-1}\Delta_H(r^{-1}h)^{-\half}\\
    &\hskip3in\rho(r)^{-1}d\nu(h)d\nu(r)\\
    &=\bip(\rho_{eG_{P}}(b^*c)(f\tensor_{eG_{P}}\xi)|
    g\tensor_{eG_{P}}\eta)_{eG_{P}}\\
    &=\bip(c\tensor f\tensor_{eG_{P}}\xi|b\tensor
    g\tensor_{eG_{P}}\eta).
  \end{align*}
  Clearly $U$ intertwines $\ibind{B(x)}(\rho_{eG_{P})})$ with
  $\rho_{xG_{P}}$.

  Notice that $\set{c\tensor f(h)\,:\,f\in \Gamma_c(G,\B), c\in B(x)}$
  is dense in $B(x)\tensor_A B(h)\simeq B(xh)$. Therefore the set of
  $f_c(xh)$ where $f_c(xh)=cf(h)$, $f\in \Gamma_c(G,\B)$, and $c\in
  B(x)$ is dense in $B(xh)$. \cite{muhwil:dm08}*{Lemma~A.4} implies
  that $U$ has dense range.
\end{proof}

Of course, just as with group dynamics (see
\cite{echwil:tams08}*{\S2.3}), the homogeneity of $\pi\restr A$ will
play a crucial role.  As in \cite{echwil:tams08}*{Remark~1.5}, we will
use \cite{sau:jfa79}*{Lemme~1.5} or \cite{eff:tams63}*{Theorem~1.4} to
characterize homogeneous representations $\rho$ of $A$ as those that
have the property that for any ideal $I$ in $A$ such that $I
\not\subset \ker\rho$,
$\overline{\rho(I)\HH_{\rho}}=\HH_{\rho}$. Alternatively, note that
$\rho$ is homogeneous if given any approximate identity $\set{u_{i}} $
for an ideal $I\not\subset \ker \rho$, we have $\rho(u_{i})\to
I_{\HH_{\rho}}$ in the strong operator topology.

We need the following general Morita Equivalence result which we have not seen
elsewhere.\footnote{For other such results, see \cite{hrw:pams07}.}
\begin{prop}
  \label{prop-ind-homo}
  Let $A$ and $B$ be $\cs$-algebras and let $\X$ be an $A\sme B$-\ib.
  Then $\pi$ is a homogeneous representation of $B$ if and only if
  $\xind \pi$ is a homogeneous representation of $A$.
\end{prop}
\begin{proof}
  By symmetry, it suffices to show that $\pi$ homogeneous implies that
  $\xind \pi$ is homogeneous.  So, we suppose that $\pi$ is
  homogeneous and that $K$ is an ideal in $A$ with $K\not\subset \ker
  (\xind \pi )$.  Let $\set{u_{i}}$ be an approximate unit for $K$.
  As remarked above, it will suffice to see that $(\xind
  \pi)(u_{i})\to I$ in the strong operator topology.  By the Rieffel
  correspondence (\cite{rw:morita}*{\S3.3}), we can assume that
  $K=\xind J$ for an ideal $J$ in $B$ with $J\not\subset \ker\pi$.
  Since $\pi$ is homogeneous, $\overline{\pi(J)\HH_{\pi}}=\HH_{\pi}$.
  Hence it will suffice to show that
  \begin{equation*}
    u_{i}\cdot x \tensor \pi(b)h \to x\tensor \pi(b)h
  \end{equation*}
  in the Hilbert space $\X\tensor_{B}\HH_{\pi}$ for all $x\in \X$,
  $b\in J$ and $h\in \HH_{\pi}$.

  However, $\X\cdot J$ is a $K\sme J$-\ib\ (see
  \cite{rw:morita}*{Proposition~3.25}).  Hence for any $y\in \X\cdot
  J$, we have $u_{i}\cdot y\to y$ in $\X$ (for example, see
  equation~(2.5) in the proof of \cite{rw:morita}*{Corollary~2.7}).
  So in $\X\tensor_{B}\HH_{\pi}$, we have
  \begin{align*}
    \bip(u_{i}\cdot x\tensor\pi(b)h| {u_{i}\cdot x\tensor\pi(b)h}) &=
    \bip(\pi\bigl(\rip\star<u_{i}\cdot x,u_{i}\cdot
    x>\bigr)\pi(b)h|{\pi(b)h}) \\
    &= \bip(\pi\bigl(\rip\star<u_{i}\cdot x\cdot b,u_{i}\cdot x\cdot
    b>\bigr)h|h),
  \end{align*}
  which converges to
  \begin{equation*}
    \bip(\pi\bigl(\rip\star<x,x>\bigr)\pi(b)h|{\pi(b)h}) =
    \bip(x\tensor\pi(b)h | {x\tensor\pi(b)h})
  \end{equation*}
  since $x\cdot b\in \X\cdot J$.  Then, using similar calculations, we
  can see that
  \begin{align*}
    \|u_{i}\cdot x\tensor\pi(b)h& -x\tensor \pi(b)h\|^{2} \\
    &= \bip(u_{i}\cdot x\tensor\pi(b)h -x\tensor \pi(b)h |{u_{i}\cdot
      x\tensor\pi(b)h -x\tensor \pi(b)h})
  \end{align*}
  converges to zero as required.
\end{proof}


\begin{prop}
  \label{prop-key-homo}
  Suppose that $L$ has kernel $P$ and that $\pi\restr A$ is
  homogeneous.  Then, for all $x\in G$, $\rho_{xG_{P}}$ is a
  homogeneous representation of $A$ with kernel $x\cdot P$.
\end{prop}
\begin{proof}
  By Lemma~\ref{lem-equiv-pi} and our assumptions on $L$ and $\pi$, we
  have that $\rho_{eG_{P}}$ is homogeneous with kernel $P$.  The rest
  follows from Lemma~\ref{lem-key-induced} and
  Proposition~\ref{prop-ind-homo}.
\end{proof}

We have set things up so that the representation of $A$ given by
$\Pihat\restr A$ (where $\Pihat$ is defined in \eqref{pihat}) is the
direct integral
\begin{equation}\label{eq:14}
  \int_{G/G_{P}}^{\oplus} \rho_{xG_{P}} \,d\mugmh
\end{equation}
on $L^{2}(G/G_{P}*\VV,\mugmh)$.  Recall that the \emph{diagonal
  operators}, $\Delta(G/G_{P}*\VV,\mugmh)$, on
$L^{2}(G/G_{P}*\VV,\mugmh)$ are, by definition, the multiplication
operators determined by bounded Borel functions on $G/G_{P}$ (see
\cite{wil:crossed}*{Definition~F.13}).  In this case,
$\Delta(G/G_{P}*\VV,\mugmh)$ are exactly the operators of the form
$M(\phi)$ with $\phi\in L^{\infty}(G/G_{P})$ and
\begin{equation*}
  M(\phi)(f\tensor \xi)=\phi\cdot f\tensor \pi,
\end{equation*}
where $\phi\cdot f(x)=\phi(x G_{P})f(x)$.

We have worked fairly hard to see that \eqref{eq:14} is an ideal
center decomposition (see \cite{wil:crossed}*{Definition~G.18 and
  Theorem~G.20}).  As in the proof of
\cite{echwil:tams08}*{Theorem~1.7}, the essential feature we require
from this observation is that $M(L^{\infty}(G/G_{P}))$ lies in the
center of the commutant of $\Pihat(A)$ so that
\begin{equation}
  \label{eq:15}
  \Pihat(A)'\subset M(L^{\infty}(G/G_{P}))'\subset M(C_{0}(G/G_{P}))'.
\end{equation}



\begin{lemma}
  \label{lem-dense}
  Let $f\in \sa_{c}(G,\B)$ and $\phi\in C_{c}(G/G_{P})$.  Then we get
  a section $f\tensor \phi$ in $\sa_{c}(G\times
  G/G_{P},\pr_{1}^{*}\B)$ via
  \begin{equation*}
    f\tensor\phi(y,xH)=\phi(xH)f(y).
  \end{equation*}
  Such sections span a dense subspace of $\sa_{c}(G\times
  G/G_{P},\pr_{1}\B)$ in the inductive limit topology.
\end{lemma}
\begin{proof}
  Let $\mathcal{G}$ be the span of sections of the form $f\tensor
  \phi$. Since $g\cdot (f\tensor\phi)=(g\cdot f)\tensor\phi$ and
  $\psi\cdot (f\tensor\phi)=f\tensor(\psi\cdot\phi)$ for $g\in
  C_c(G),f\in \sa_C(G,\B)$, $\psi,\phi\in C_c(G/G_P)$, it follows that
  $\mathcal{G}$ is closed under multiplication by functions in
  $C_{c}(G\times G/G_{P})$.

  Notice that, for $(y,xH)\in G\times G/G_{P}$, we have that
  $\pr_{1}^{*}\B(y,xH)=B(y)$. Therefore, the set $\set{f\tensor
    \phi(y,xH)\: | f\tensor\phi\in \mathcal{G}}$ is dense in
  $\pr_{1}^{*}\B(y,xH)$.  A small variation of
  \cite{muhwil:dm08}*{Lemma~A.4} implies the result.
\end{proof}

\begin{proof}[Proof of Theorem~\ref{thm-main-1.7} in the case $G$ is a
  group] 
  Suppose that $T\in B(L^{2}(G/G_{P}*\VV,\mugmh)$ is in the commutant
  of $\Ind_{G_{P}}^{G }L$.  It suffices to see that $T$ is a scalar
  operator.
  But $T$ must also commute with $\Pihat(b)=(\Ind_{G_{P}}^{G}
  L)\bigl(\iota_{\B}(b)\bigr)$ for all $b\in \B$.  In particular, $T$
  belongs to $\Pihat(A)'$.  Hence $T$ commutes with $M(\phi)$ for all
  $\phi\in C_{c}(G/G_{P})$ by \eqref{eq:15}.  But
  \begin{equation*}
    (\xind L)(f\tensor \phi)= M(\phi)(\Ind_{G_{P}}^{G} L)(f).
  \end{equation*}
  Thus $T$ commutes with $(\xind L)(f\tensor \phi)$.  By
  Lemma~\ref{lem-dense}, this means $T$ commutes with the irreducible
  representation $\xind L$.  Thus $T$ is a scalar.
\end{proof}

\section{The Type I Case}
\label{sec:type-i-case}


In this section, we prove the appropriate analogue of
\cite{echwil:tams08}*{Lemma~3.2}.  Recall that if $p:\B\to G$ is a Fell
bundle and if $P$ is a primitive
ideal of $A=\sa_{0}(\go,\B)$, then $P\supset I(u)$ for
$u=\sigma_{A}(P)$.  In particular, we can view $P$ as a primitive
ideal of the quotient $A(u)=B(u)$.  Of course, $P$ is locally closed
in $\prima$ if and only if it is locally closed when viewed as a
primitive ideal of $A(u)$.
\begin{prop}
  \label{prop-key-loc-closed}
  Let $p:\B\to G$ be a separable, saturated Fell bundle over a locally
  compact groupoid $G$ and let $P$ be a primitive ideal in $\Prim A$
  where $A=\sa_{0}(\go,\B)$ is the associated \cs-algebra.  Suppose
  that $P$ is locally closed in $\Prim A$ and 
  that $L$ is an irreducible representation of $\cs(G_{P},\B)$ which
  is the integrated form of the strongly continuous map
  $\pi:\B\restr{G_{P}}\to B(\H)$ such that $\ker \pi\restr{A(u)}=P$
  where $u=\sigma_{A}(P)$. 
  Then $\pi\restr{A(u)}$ is a homogeneous representation of $A(u)$.
\end{prop}

We'll need the following standard observation from Morita theory.

\begin{lemma}
  \label{lem-hx-int}
  Suppose that $\X$ is an $A\sme B$-\ib.  Then the Rieffel
  correspondence $h:\I(B)\to\I(A)$ preserves arbitrary intersections.
\end{lemma}
\begin{proof}
  Notice that
  \begin{equation*}
    h(I)=\set{a\in J:a\cdot \X\subset \overline{\X\cdot I}}.
  \end{equation*}
Hence the result follows once we establish that
\begin{equation*}
  \bigcap \overline{\X\cdot I_{J}}=\overline{\X\cdot \bigcap I_{J}}.
\end{equation*}
But this can be proved exactly as in the proof of
\cite{wil:crossed}*{Lemma~5.19}. 
\end{proof}

\begin{proof}[Proof of Proposition~\ref{prop-key-loc-closed}]
  We may as well replace $\B$ by $\B\restr{G_{P}}$ and $A$ by $A(u)$.
   Let $\rho=\pi\restr{A}$.  As in the proof of
  \cite{echwil:tams08}*{Lemma~3.2}, it will suffice to see that if $I$
  is any ideal in $A$ such that $I\not\subset P$, then
  $\overline{\rho(I)\H} =\H$.  Since $\rho(I)=\rho(I+P)$, we can
  assume $I$ properly contains $P$.  If $x\in G_{P}$, let $x\cdot
  I:=h_{x}(I)$.  Then $x\cdot I$ properly contains $P$.  Since any
  ideal in $A$ is the intersection of those primitive ideals which
  contain it,
  \begin{equation*}
    K:=\bigcap_{x\in G_{P}} x\cdot I \supset \bigcap
    _{\substack{P'\in\Prim A \\ P'\supset P \\ P'\not=P}} P':= J.
  \end{equation*}
Since $P$ is locally closed in $\Prim A$,
\cite{echwil:tams08}*{Lemma 3.1} implies that $J$ properly contains
$P$.  Hence so does $K$.  Moreover $K$ must be $G_{P}$-invariant as
defined in \cite{ionwil:hjm11}*{\S3.1}: if $P'$ is a primitive ideal
containing $K$, then using Lemma~\ref{lem-hx-int}, we have
\begin{equation*}
  y\cdot P' = h_{y}(P') \supset \bigcap_{x\in G_{P}} h_{y}(x\cdot I) =
  \bigcap _{y\in G_{P}}h_{xy}(I) =K.
\end{equation*}

Thus we can identify $\cs(G_{P},\B_{K})$ with an ideal of $\cs(G_{P},\B)$ as
in \cite{ionwil:hjm11}*{Lemma~3.5}.  We claim
\begin{equation}
  \label{eq:8}
  L\bigl(\cs(G_{P},\B_{K})\bigr)\H\subset \overline{\pi(K)\H}.
\end{equation}
To see this, let $w\in (\pi(K)\H)^{\perp}$.  Then for any $h\in\H$,
\begin{equation}\label{eq:21}
  \bip(L(f)h|w)=\int_{G_{P}} \Delta_{G_{p}}(s)^{-\half}\bip(\pi(f(s))h|w)\,ds.
\end{equation}
But $f(s)\in B(s)\cdot K=(s\cdot K)\cdot B(s)=K\cdot B(s)$.  Hence
$\pi(f(s))h\in \pi(K)\H$ for each $h$.  Thus the integrand in \eqref{eq:21} is
zero and $\bip(L(f)h|w)=0$ for all $w\in (\pi(K)\H)^{\perp}$.  Thus
\eqref{eq:8} follows.

Next we claim that
\begin{equation}
  \label{eq:22}
  L\bigl(\cs(G_{P},\B_{K})\bigr)\H\not=\sset0.
\end{equation}
To see this, note that as $K$ properly contains $P$, there is an $a\in
K$ and $h\in\H$ such that $\pi(a)h\not=0$.  But as Fell bundles always
have sufficiently many sections, there is an $f\in
\sa_{c}(G_{P},\B_{K})$ such that $f(e)=a$.  Since $\pi$ is strongly
continuous, there is a neighborhood $V$ of $e$ in $G_{P}$ such that
for all $s\in V$ we have
\begin{equation*}
  \|\pi(f(s))h- \pi(a)h\|<\frac14 \|\pi(a)h\|,
\end{equation*}
and such that
\begin{equation*}
  \Delta_{G_{P}}(s)^{-\half}  \le 2.
\end{equation*}
Let $\phi\in C_{c}^{+}(G_{P})$ be such that $\supp \phi\subset V$ and
\begin{equation*}
  \int_{G_{P}}\Delta_{G_{P}}(s)^{-\half}\phi(s)\,ds=1.
\end{equation*}
Then $\phi\cdot f\in \sa_{c}(G_{P},\B_{K})$ and
\begin{align*}
  \| L(\phi\cdot f)h-\pi(a)f\| &= \Bigl\| \int_{G_{P}}
  \Delta_{G_{P}}(s)^{-\half} \phi(s) \bigl(\pi(f(s)h-\pi(a)h\bigr)
  \,ds\Bigr\| \\
&< 2\frac 14 \|\pi(a)h\|=\frac12\|\pi(a)h\|.
\end{align*}
Thus $L(\phi\cdot f)h\not=0$.  This establishes \eqref{eq:22}.

However, as $L$ is assumed irreducible, \eqref{eq:22} and \eqref{eq:8}
imply that
\begin{equation*}
  \overline{L\bigl(\cs(G_{P},\B_{K})\bigr)\H} =\H \subset
  \overline{\pi(K)\H}. 
\end{equation*}
Thus $\overline{\rho(I)\H}=\H$ and $\rho$ is homogeneous as claimed.
\end{proof}

Theorems \ref{prop-key-loc-closed} and \ref{thm-main-1.7} imply
immediately the following generalization of
\cite{echwil:tams08}*{Theorem 1.8}.
\begin{prop}
  Suppose that $p:\B\to G$ is a separable, saturated Fell bundle over
  a locally compact groupoid $G$ such that
  all points are locally closed in $\Prim A$. Then $p:\B\to G$
  satisfies strong-EHI. In particular, if $A$ is of type I, then
  $p:\B\to G$ satisfies the strong-EHI.
\end{prop}

\begin{remark}
  \label{rem-loc-closed}
  Even if $A$ is not of type~I, there are a number of general
  situations where $\prima$ necessarily has locally closed points.
  For example, as
  we observe in \cite{echwil:tams08}*{Proposition~1.8}, this is case
  in any subquotient of the group \cs-algebra of an almost connected
  Lie group.  Nevertheless, there are separable \cs-algebras for which
  this property fails (see \cite{echwil:tams08}*{Example~3.3}).
\end{remark}

\section{Examples}
\label{sec:examples}

\subsection{Groupoid dynamical systems}
\label{sec:groupoid-dynamical-systems}

 Let $G$ be
a locally compact Hausdorff groupoid with Haar system
$\set{\lambda^u}_{u\in \go}$. Let
$\pi:\A\to \go$ be an upper semicontinuous $\cs$-bundle over $\go$ and
let $A=\Gamma_0(\go,\A)$.
Assume that $(\A,G,\alpha)$ is a groupoid dynamical system (see, for
example, \cite{muhwil:nyjm08}*{Definition 4.1}). Recall
from \cite{muhwil:dm08}*{\S2} that one can define a Fell bundle
$p:\B\to G$, where $\B:=r^{*}\A=\set{(a,x):\pi(a)=r(x)}$ is the
pull-back of $\A$ via the range map $r$. The multiplication on $\B$ is
defined via
\[
(a,x)(b,y):=(a\alpha_x(b),xy),
\] if $(x,y)\in G^{(2)}$, and the involution is given by
\[
(a,x)^*:=(\alpha_x^{-1}(a^*),x^{-1}).
\]
Then $\Gamma_c(G,r^{*}\A)$ is a $*$-algebra with respect to the
operations
\[
f*g(x)=\int_Gf(y)\alpha_y\bigl(g(y^{-1}x)\bigr)\,d\lambda^{r(x)}(y)
\quad\text{and}\quad
f^*(x)=\alpha_x\bigl(f(x^{-1})^*\bigr).
\]
The \emph{crossed product} $\A\rtimes_\alpha G$ is the enveloping
$\cs$-algebra of $\Gamma_c(G;r^*\A)$.

Recall (see, for example, \cite{muhwil:nyjm08}*{Definition 7.5}) that
a unitary representation of a groupoid $G$ with Haar system
$\set{\lambda^u}_{u\in\go}$ is a triple $(\mu,\go*\HH,L)$ consisting
of a quasi-invariant measure $\mu$ on $\go$, a Borel Hilbert bundle
$\go*\HH$ over $\go$ and a Borel homomorphism $\tilde{U}:G\to
\operatorname{Iso}(\go*\HH)$ such that
\[
\tilde{U}(x)=(r(x),L_x,s(x)),
\]
where $\operatorname{Iso}(\go*\HH)$ is the \emph{isomorphism
  groupoid}. 

A covariant representation $(M,\mu,\go*\HH,U)$ of $(\A,G,\alpha)$
(\cite{muhwil:nyjm08}*{Definition 7.9}) consists of a unitary
representation $(\mu,\go*\HH,U)$ of $G$ and a $C_0(\go)$-linear
representation $M:A\to B(L^2(\go*\HH,\mu))$ such that there are
representations $M_u:A\to B(\HH(u))$ so that
\[
M(a)h(u)=M_u(a)(h(u))\text{ for $\mu$-almost all }u,
\]
and such that there is a $\nu:=\lambda\circ \mu$-null set $N$ such
that for all $x\notin N$,
\[
U_xM_{s(x)}(b)=M_{r(x)}(\alpha_x(b))U_x\text{ for all }b\in A(s(x)).
\]
Recall from \cite{muhwil:nyjm08}*{Proposition 7.11} that if
$(M,\mu,\go*\HH,U)$ is a covariant representation of $(\A,G,\alpha)$,
then there is a $\Vert \cdot \Vert_I$-norm decreasing
$*$-representation $L$ of $\Gamma_c(G;r^*\A)$ called the \emph{integrated
form} of the covariant representation  given by
\[
L(f)h(u)=\int_GM_u(f(x))U_xh(s(x))\Delta(x)^{-\frac{1}{2}}d\lambda^u(x).
\]
Conversely, given any representation $L$ of $\A\rtimes_\alpha G$,
there is a covariant representation $(M,\mu,\go*\HH,U)$ such that $L$
is equivalent to the corresponding integrated form
(\cite{muhwil:nyjm08}*{Theorem 7.12}).

Notice that if $G$ is a group and $(\rho,U)$ is a covariant
representation of $(A,G,\alpha)$, then, since we are treating $G$ as a
groupoid, the corresponding integrated form is given by the formula
\[
\rho\rtimes U(f)=\int_G\rho(f(s))U_s\Delta(s)^{-\frac{1}{2}}d\mu(s)
\]
(compare this formula against, for example,
\cite{wil:crossed}*{Equation (2.19)}).

If $H$ is a closed subgroupoid of $G$ and $L$ is a representation of
$\A\rtimes_{a|_H}H$, then the construction of induced representations
from \cite{simwil:jot11}*{\S4.1} (see also Section \ref{sec:step-i}
above) gives us that the induced representation $\Ind_H^GL$ acts on
the completion of $X\odot \HH_L$ by
\[
(\Ind_H^GL)(f)(\varphi\tensor h)=f*\varphi\tensor h,
\]
where $f*\varphi(z)=\int_Gf(y)\alpha_y\bigl(\varphi(y^{-1}z)\bigr)\,d\lambda^{r(z)}(y)$.
Our Theorem \ref{thm:indirr} seems to be new for this set-up.
\begin{thm}
  \label{thm:irrrep_gcp}
  Assume that $(\A,G,\alpha)$ is a groupoid dynamical system. Let
  $u\in \go$ and suppose that $L$ is an irreducible representation 
  of $\A\rtimes_{\alpha|_{G(u)}}G(u)$. Then $\Ind_{G(u)}^G L$ is an
  irreducible representation of $\A\rtimes_{\alpha}G$. 
\end{thm}
Recall from \cite{ionwil:hjm11}*{Remark 3} that the $G$ action on
$\Prim A$ is the same as the usual one: $x\cdot P=\alpha_x(P)$. Our
main theorem (Theorem \ref{thm-main-1.7}) becomes: 
\begin{thm}
  Let $(\A,G,\alpha)$ be a groupoid dynamical system. Let $P\in \Prim
  A$ and let $(\rho,U)$ be a covariant representation of
  $(A,G_P,\alpha_{|_{G_P}})$. Assume that $\ker\rho=P$ and that
  $\rho\rtimes U$ is irreducible. If either $A$ is type~I or if $\rho$ is
  homogeneous, then $\Ind_{G_{P}}^G (\rho\rtimes U)$
  is irreducible.
\end{thm}
In particular we recover the main result of \cite{echwil:tams08}.

\subsection{Green-Renault's twisted groupoid dynamical systems}
\label{sec:green-rena-group}

Suppose that
\[
\go\to S\xrightarrow{i}\Sigma\xrightarrow{j}G\rightarrow \go
\]
is a groupoid extension of locally compact groupoids over $\go$ where
$S$ is a group bundle of \emph{abelian} groups admitting a Haar
system. We view $S$ as a closed subgroupoid of $\Sigma$.  We assume
that we have a groupoid dynamical system $(\A,\Sigma,\alpha)$, so that
$\pi:\A\to\go=\Sigma^{(0)}$ is an upper semicontinuous
$\cs$-bundle. We also need an element $\chi\in \prod_{s\in
  S}M(A(r(s)))$ such that
\[
(s,a)\mapsto \chi(s)a
\]
is continuous from $S*\A$ to $\A$, and such that
\[
\alpha_s(a)=\chi(s)a\chi(s)^*\;\text{for all }\;(s,a)\in S*\A,
\]
and
\[
\chi(\sigma s\sigma^{-1})=\overline{\alpha}_\sigma(\chi(s))\;\text{
  for }\;(\sigma,s)\in \Sigma^{(2)}.
\]
Following \cite{ren:jot87} and \cite{ren:jot91}, we call
$(G,\Sigma,\A)$ a twisted groupoid dynamical system.  As in
\cite{muhwil:dm08}*{Example 2.5}, we define an $S$-action on
$r^*\A=\{(a,\sigma)\,:\,\pi(a)=r(\sigma)\}$ by
\[
(a,\sigma)\cdot s:=(a\chi(s)^*,s\sigma).
\]
The associated Fell bundle is then $\B:=r^*\A/S$ with the map $p:\B\to
G$, $p([a,\sigma])=j(\sigma)$. The operations in $\B$ are defined via
(see \cite{muhwil:dm08}*{Example 2.5} for details):
\[ [a,\sigma][b,\tau]:=[a\alpha_\sigma(b),\sigma\tau]
\]
if $(j(\sigma),j(\tau))\in G^{(2)}$ and
\[ [a,\sigma]^*=[\alpha_\sigma^{-1}(a^*),\sigma^{-1}].
\]

To define a section of $\B$, we need a continuous function
$f:\Sigma\to \A$ such that $f(\sigma)\in A(r(\sigma))$ and
$f(s\sigma)=f(\sigma)\chi(s)^*$ for all $s\in S$ and $\sigma\in
\Sigma$ such that $(s,\sigma)\in \Sigma^{(2)}$. The corresponding
section is given by
$\check{f}(j(\sigma))=[f(\sigma),\sigma]$. Replacing $\check{f}$ with
$f$, the $*$-operations on $\Sigma$ are given by
\[
f*g(\sigma)=\int_Gf(\tau)\alpha_{\tau}\bigl(g(\tau^{-1}\sigma)\bigr)\,d\lambda^{r(j(\sigma))}(\tau)\quad\text{and}\quad
f^*(\sigma)=\alpha_\sigma(f(\sigma^{-1})^*).
\]
As described in \cite{muhwil:dm08}*{Example 2.10}, the completion is
Renault's $\cs(G,\Sigma,\A,\lambda)$ from \cite{ren:jot87} and
\cite{ren:jot91}.

A covariant representation $(M,\mu,\go*\HH,U)$ of $(G,\Sigma,\A)$
(\cite{ren:jot87}*{Definition 3.4}) consists of a unitary
representation $(\mu,\go*\HH,U)$ of $\Sigma$ and a $C_0(\go)$-linear
representation of $M:A\to B(L^2(\go*\HH),\mu)$ so that there are
representations $M_u:A\to B(H(u))$ such that
\[
M(a)h(u)=M_u(a)(h(u))\;\text{for $\mu$-almost all }u,
\]
and such that there is a $\mu$-conull set $V$ such that
\[
U_xM_{s(x)}(a)=M_{r(x)}(\alpha_x(a))U_x\;\text{for all }x\in \Sigma_V
\text{ and } a\in A(s(x)),
\]
and
\[
U_s=M_{r(s)}(\chi(s))\;\text{for all}\; s\in S_V.
\]
If $(M,\mu,\go*\HH,U)$ is a covariant representation of
$(G,\Sigma,\A)$, then \cite{ren:jot87}*{Proposition 3.5} (see also
\cite{muhwil:dm08}*{Proposition 4.10}) implies that there is a $\Vert
\cdot\Vert_I$-norm decreasing $*$-representation $L$ of
$\Gamma_c(G,\Sigma,\A)$, called the integrated form of the covariant
representation, given by
\[
L(f)h(u)=\int_GM_u(f(x))U_xh(s(x))\Delta(j(x))^{-1/2}d\lambda^u(j(x)).
\]
Conversely, \cite{ren:jot87}*{Theorem 4.1} and
\cite{muhwil:dm08}*{Theorem 4.13} imply that every representation of
$C^*(G,\Sigma,\A)$ is equivalent to the integrated form of a covariant
representation.

Notice that if $\Sigma$ is a group and $S$ is an abelian subgroup of
$\Sigma$ then $G=\Sigma/S$. We recover the twisted dynamical systems
of Green \cite{gre:am78} and \DH{\d{\u{a}}}ng Ng\d oc
\cite{ngo:77}. If $(\rho,U)$ is 
a covariant representation of $(A,G,\alpha)$ that preserves $\chi$,
i.e. $\overline{\rho}(\chi(s))=U(s)$ for all $s\in S$, then, since we
are treating groups as groupoids, the integrated form
$\rho\rtimes^\chi U$ of $(\rho,U)$ is
\[
\rho\rtimes^\chi U(f)=
\int_G\rho(f(s))U(s)\Delta(\dot{s})^{-\frac{1}{2}}d\mu(\dot{s}).
\]
If $H$ is a closed subgroupoid of $G$ and $L$ is a representation of
$\cs(H,\Sigma,\A)$, then the induced representation $\Ind_H^G$ acts on
the completion of $X\odot \HH_L$ (see Section \ref{sec:step-i} for
details) via
\[
\Ind_H^G(L)(f)(\phi\tensor h)=f*\phi\tensor h,
\]
where
$f*\phi(z)=\int_Gf(y)\alpha_y(g(y^{-1}z))\,d\lambda^{r(j(z))}(y)$. As
an immediate consequence of Theorem \ref{thm:indirr} we obtain the
following:
\begin{thm}
  Assume that $(G,\Sigma,\A)$ is a twisted groupoid dynamical
  system. Let $u\in\go$ and suppose that $L$ is an irreducible
  representation of $\cs(G(u),\Sigma(u),\A)$. Then $\Ind_{G(u)}^GL$ is
  an irreducible representation of $\cs(G,\Sigma,\A)$.
\end{thm}
Using \cite{ionwil:hjm11}*{Lemma 2.1} one obtains that the $G$ action
on $\Prim A$ is the same as the usual one. This fact, together with
\ref{thm-main-1.7}, implies the following:
\begin{thm}
  Let $(G,\Sigma,\A)$ be a twisted groupoid dynamical system. Let
  $P\in \Prim A$ and let $(\rho,U)$ be a covariant representation of
  $(G_P,\Sigma_P,\A)$ that preserves $\chi$ and is such that
  $\ker\rho=P$ and $\rho\rtimes^\chi U$ is irreducible. If either $A$
  is of type~I or if $\rho$ is homogeneous, then $\Ind_{G_{P}}^G
  \rho\rtimes^\chi U$ is irreducible.
\end{thm}

\bibliographystyle{amsxport}
\bibliography{references-nov01}

\def\noopsort#1{}\def\cprime{$'$} \def\sp{^}
\begin{bibdiv}
\begin{biblist}

\bib{bmz:pems13}{article}{
      author={Buss, Alcides},
      author={Meyer, Ralf},
      author={Zhu, Chenchang},
       title={A higher category approach to twisted actions on
  {$C^*$}-algebras},
        date={2013},
        ISSN={0013-0915},
     journal={Proc. Edinb. Math. Soc. (2)},
      volume={56},
      number={2},
       pages={387\ndash 426},
         url={http://dx.doi.org/10.1017/S0013091512000259},
      review={\MR{3056650}},
}

\bib{echwil:tams08}{article}{
      author={Echterhoff, Siegfried},
      author={Williams, Dana~P.},
       title={Inducing primitive ideals},
        date={2008},
     journal={Trans. Amer. Math. Soc},
      volume={360},
       pages={6113\ndash 6129},
}

\bib{eff:tams63}{article}{
      author={Effros, Edward~G.},
       title={A decomposition theory for representations of {$C^*$}-algebras},
        date={1963},
     journal={Trans. Amer. Math. Soc.},
      volume={107},
       pages={83\ndash 106},
      review={\MR{26 \#4202}},
}

\bib{gre:am78}{article}{
      author={Green, Philip},
       title={The local structure of twisted covariance algebras},
        date={1978},
     journal={Acta Math.},
      volume={140},
       pages={191\ndash 250},
}

\bib{hrw:pams07}{article}{
      author={Huef, Astrid~an},
      author={Raeburn, Iain},
      author={Williams, Dana~P.},
       title={Properties preserved under {M}orita equivalence of {$C\sp
  *$}-algebras},
        date={2007},
        ISSN={0002-9939},
     journal={Proc. Amer. Math. Soc.},
      volume={135},
      number={5},
       pages={1495\ndash 1503},
      review={\MR{MR2276659}},
}

\bib{ionwil:pams08}{article}{
      author={Ionescu, Marius},
      author={Williams, Dana~P.},
       title={Irreducible representations of groupoid {$C\sp *$}-algebras},
        date={2009},
        ISSN={0002-9939},
     journal={Proc. Amer. Math. Soc.},
      volume={137},
      number={4},
       pages={1323\ndash 1332},
      review={\MR{MR2465655}},
}

\bib{ionwil:hjm11}{article}{
      author={Ionescu, Marius},
      author={Williams, Dana~P.},
       title={Remarks on the ideal structure of {F}ell bundle
  {$C^*$}-algebras},
        date={2012},
     journal={Houston J. Math.},
      volume={38},
       pages={1241\ndash 1260},
}

\bib{kmqw:nyjm10}{article}{
      author={Kaliszewski, S.},
      author={Muhly, Paul~S.},
      author={Quigg, John},
      author={Williams, Dana~P.},
       title={Coactions and {F}ell bundles},
        date={2010},
        ISSN={1076-9803},
     journal={New York J. Math.},
      volume={16},
       pages={315\ndash 359},
         url={http://nyjm.albany.edu:8000/j/2010/16_315.html},
      review={\MR{2740580}},
}

\bib{muh:cm01}{incollection}{
      author={Muhly, Paul~S.},
       title={Bundles over groupoids},
        date={2001},
   booktitle={Groupoids in analysis, geometry, and physics ({B}oulder, {CO},
  1999)},
      series={Contemp. Math.},
      volume={282},
   publisher={Amer. Math. Soc.},
     address={Providence, RI},
       pages={67\ndash 82},
      review={\MR{MR1855243 (2003a:46085)}},
}

\bib{muhwil:dm08}{article}{
      author={Muhly, Paul~S.},
      author={Williams, Dana~P.},
       title={Equivalence and disintegration theorems for {F}ell bundles and
  their {$C\sp *$}-algebras},
        date={2008},
        ISSN={0012-3862},
     journal={Dissertationes Math. (Rozprawy Mat.)},
      volume={456},
       pages={1\ndash 57},
      review={\MR{MR2446021}},
}

\bib{muhwil:nyjm08}{book}{
      author={Muhly, Paul~S.},
      author={Williams, Dana~P.},
       title={Renault's equivalence theorem for groupoid crossed products},
      series={NYJM Monographs},
   publisher={State University of New York University at Albany},
     address={Albany, NY},
        date={2008},
      volume={3},
        note={Available at http://nyjm.albany.edu:8000/m/2008/3.htm},
}

\bib{ngo:77}{unpublished}{
      author={Ng{\d o}c, Nghi{\^e}m~{\DH}{\d{\u{a}}}ng},
       title={Produits crois\'es restreints et extensions of groupes},
        date={1977},
        note={Notes, Paris},
}

\bib{rw:morita}{book}{
      author={Raeburn, Iain},
      author={Williams, Dana~P.},
       title={Morita equivalence and continuous-trace {$C^*$}-algebras},
      series={Mathematical Surveys and Monographs},
   publisher={American Mathematical Society},
     address={Providence, RI},
        date={1998},
      volume={60},
        ISBN={0-8218-0860-5},
      review={\MR{2000c:46108}},
}

\bib{ren:jot87}{article}{
      author={Renault, Jean},
       title={Repr\'esentation des produits crois\'es d'alg\`ebres de
  groupo\"\i des},
        date={1987},
        ISSN={0379-4024},
     journal={J. Operator Theory},
      volume={18},
      number={1},
       pages={67\ndash 97},
      review={\MR{MR912813 (89g:46108)}},
}

\bib{ren:jot91}{article}{
      author={Renault, Jean},
       title={The ideal structure of groupoid crossed product \cs-algebras},
        date={1991},
     journal={J. Operator Theory},
      volume={25},
       pages={3\ndash 36},
}

\bib{sau:jfa79}{article}{
      author={Sauvageot, Jean-Luc},
       title={Id\'eaux primitifs induits dans les produits crois\'es},
        date={1979},
        ISSN={0022-1236},
     journal={J. Funct. Anal.},
      volume={32},
      number={3},
       pages={381\ndash 392},
      review={\MR{81a:46080}},
}

\bib{simwil:jot11}{article}{
      author={Sims, Aidan},
      author={Williams, Dana~P.},
       title={Renault's equivalence theorem for reduced groupoid
  {$C^\ast$}-algebras},
        date={2012},
        ISSN={0379-4024},
     journal={J. Operator Theory},
      volume={68},
      number={1},
       pages={223\ndash 239},
      review={\MR{2966043}},
}

\bib{simwil:ijm13}{article}{
      author={Sims, Aidan},
      author={Williams, Dana~P.},
       title={Amenability for {F}ell bundles over groupoids},
        date={2013},
     journal={Illinios J. Math.},
       pages={in press},
        note={(arXiv:math.OA.1201.0792)},
}

\bib{simwil:nyjm13}{article}{
      author={Sims, Aidan},
      author={Williams, Dana~P.},
       title={An equivalence theorem for reduced {F}ell bundle
  {$C^*$}-algebras},
        date={2013},
     journal={New York J. Math.},
      volume={19},
       pages={159\ndash 178},
}

\bib{wil:crossed}{book}{
      author={Williams, Dana~P.},
       title={Crossed products of {$C{\sp \ast}$}-algebras},
      series={Mathematical Surveys and Monographs},
   publisher={American Mathematical Society},
     address={Providence, RI},
        date={2007},
      volume={134},
        ISBN={978-0-8218-4242-3; 0-8218-4242-0},
      review={\MR{MR2288954 (2007m:46003)}},
}

\end{biblist}
\end{bibdiv}

\end{document}